\documentclass[11pt]{amsart}
\usepackage{amscd,amsmath,amssymb,amsfonts,appendix}
\usepackage{color}
\usepackage[cmtip, all]{xy}
\usepackage{graphicx}
\usepackage{enumerate}
\usepackage{relsize}
\usepackage{hyperref}

\newtheorem{theorem}[subsection]{Theorem}
\newtheorem{thm}[subsection]{Theorem}

\newtheorem{lemma}[subsection]{Lemma}
\newtheorem{fact}[subsection]{Fact}

\newtheorem{proposition}[subsection]{Proposition}
\newtheorem{prop}[subsection]{Proposition}
\newtheorem{corollary}[subsection]{Corollary}

\newtheorem*{thm*}{Theorem}
\theoremstyle{remark}

\newtheorem{subremark}[subsubsection]{Remark}

\newcommand{\wv}{{\widetilde{v}}}
\newcommand{\isoarrow}{{~\overset\sim\longrightarrow~}}

\newcommand{\ZZ}{{\mathbb Z}}

\newcommand{\CG}{{\mathcal{G}}}

\newcommand{\CA}{{\mathcal{A}}}

\newcommand{\CL}{{\mathcal{L}}}

\newcommand{\G}{{\Gamma}}
\newcommand{\cO}{{\mathcal{O}}}

\newcommand{\fg}{{\frak{g}}}

\newcommand{\ra}{{~\rightarrow~}}

\newcommand{\QQ}{{\mathbb Q}}

\newcommand{\Qlb}{{\overline{\mathbb Q}_{\ell}}}

\newcommand{\RR}{{\mathbb R}}
\newcommand{\ad}{{\mathbf A}}

\newcommand{\CC}{{\mathbb C}}

 \newcommand{\PGSp}{\mathrm{PGSp}}
 
 \DeclareMathOperator{\Tr}{Tr}
 
 \def\Oct{\mathbb{O}}

 \DeclareMathOperator{\ind}{ind}
 
 \newcommand{\GL}{\mathrm{GL}}
 
 \newcommand{\N}{\mathbb N}

 \DeclareMathOperator{\ord}{ord}
 
\begin{document}
\author{Michael Harris, Chandrashekhar B. Khare, and Jack A. Thorne}


\date{\today}

\title[Local parameterization of $G_2$]{A local 
Langlands parameterization for generic supercuspidal representations of 
$p$-adic $G_2$ \\[0.3em]\smaller{}With appendix by Gordan Savin  \medskip
\\ Une param\'etrisation de Langlands locale pour les repr\'esentations supercuspidales g\'en\'eriques du groupe $p$-adique $G_2$
\\[0.3em]\smaller{}Avec une appendice par Gordan Savin
}

\begin{abstract}
    We construct a Langlands parameterization of supercuspidal representations of $G_2$ over a $p$-adic field. More precisely, for any finite extension $K / \QQ_p$ we will construct a bijection
   $$ \CL_g : \CA^0_g(G_2,K) \longrightarrow   \CG^0(G_2,K)$$
    from the set of {\it generic} supercuspidal representations of $G_2(K)$ to the set of irreducible continuous homomorphisms $\rho : W_K \to G_2(\CC)$ with $W_K$ the Weil group of $K$. The construction of the map is simply a matter of assembling arguments that are already in the literature, together with a previously unpublished theorem of G. Savin on exceptional theta correspondences, included as an appendix. The proof that the map is a bijection is arithmetic in nature, and specifically uses automorphy lifting theorems. These can be applied thanks to a recent result of Hundley and Liu on automorphic descent from $GL(7)$ to $G_2$. 
\bigskip
\bigskip

Nous construisons une param\'etrisation de Langlands des repr\'esentations supercuspidales de $G_2$ sur un corps $p$-adique.  Plus pr\'ecis\'ement,
pour chaque extension $K/\QQ_p$ nous construisons une application bijective
   $$ \CL_g : \CA^0_g(G_2,K) \longrightarrow   \CG^0(G_2,K)$$
de l'ensemble des repr\'esentations supercuspidales {\it g\'en\'eriques} de $G_2(K)$ vers l'ensemble d'homomorphismes continus et irr\'eductibles
$\rho : W_K \to G_2(\CC)$, o\`u $W_K$ d\'esigne le groupe de Weil de $K$.   Pour construire  l'application il suffit de r\'eunir des arguments qui sont d\'ej\`a dans la litt\'erature, plus un th\'eor\`eme in\'edit de G. Savin sur les correspondances th\^eta exceptionnelle, qui est d\'emontr\'e dans un appendice \'ecrite par Savin.  La d\'emonstration de la bijectivit\'e de l'application est de nature arithm\'etique, et utilise notamment des th\'eor\`emes de rel\`evement automorphes.  Ceux-ci s'applique \`a notre probl\`eme gr\^ace \`a un r\'esultat r\'ecent de Hundley et Liu sur la d\'escente automorphe de $GL(7)$ vers $G_2$. 
\end{abstract}
 
\maketitle

\setcounter{tocdepth}{1}
\tableofcontents

\section*{Introduction}

The purpose of this article is to construct a Langlands parameterization of supercuspidal representations of $G_2$ over a $p$-adic field. More precisely, for any finite extension $K / \QQ_p$ we will construct a bijection
\[ \CL_g :  \CA^0_g(G_2,K) \rightarrow \CG^0(G_2,K) \]
from the set of generic supercuspidal representations of $G_2(K)$ to the set of  irreducible continuous homomorphisms $\rho : W_K \to G_2(\CC)$  with $W_K$ the Weil group of $K$ (more precisely, between sets of equivalence classes). The construction of the map is simply a matter of assembling arguments that are already in the literature; the article \cite{KLS10} effectively contains the construction, although it doesn't specifically point out the application to supercuspidal representations.  The proof of surjectivity is an application of a recent result of Hundley and Liu \cite{HL}, which allows us to carry out a strategy, based on automorphy lifting theorems, that was initially developed in \cite{BHKT}
as an application of Vincent Lafforgue's global parameterization of automorphic representations over function fields. The proof of injectivity also uses global arithmetic methods, including automorphy lifting theorems and the Ramanujan conjecture for self-dual, regular algebraic automorphic representations of $GL(n)$, alongside known results on liftings (especially \cite{SWe, Xu}).

The parameterization is constructed in two steps.  First, following \cite{GRS97, GS04, SWe}, among other references, we use the exceptional dual reductive pair $(G_2,PGSp(6))$ in $E_7$ to define local and global correspondences from representations of $G_2$ to representations of $PGSp(6)$.  We then lift to $Sp(6)$ and use the functorial transfer of \cite{CKPS} to obtain an automorphic representation of $GL(7)$.  Using the local Langlands correspondence for $GL(n)$, we can thus obtain
a parameterization of supercuspidal representations of $G_2$ by Galois parameters with values in $GL(7)$.  We use a global argument and Chebotarev density (following \cite{Ch}) to show that the parameter takes values in the image of $G_2$ under its $7$-dimensional irreducible representation $r_7$.

The proof of surjectivity is arithmetic. For the moment, let $K$ be a $p$-adic field and let $\rho$ be a continuous homomorphism
$$\rho:  W_K \ra G_2(\CC).$$
We assume $\rho$ is {\it irreducible}:  that its image is contained in no proper parabolic subgroup.  Since the image is finite, we may replace the coefficient field $\CC$ by a sufficiently large finite field $k$ of characteristic $\ell \neq p$.  Following Moret-Bailly we show first that $K$ may be viewed as the completion at a $p$-adic place $v$ of a totally real field $F$, and that $\rho$ can be extended to a surjective homomorphism $\operatorname{Gal}(\overline{F}/F) \ra G_2(k)$ that is {\it odd}, in an appropriate sense.  We then use  the lifting method in \cite{KW}  to lift $\rho$ to a homomorphism $\tilde{\rho}:  \operatorname{Gal}(\overline{F}/F)\ra G_2(W(k))$ in such a way that  $r_7\circ \tilde{\rho}$ is geometric, in the sense of Fontaine--Mazur, and Hodge--Tate regular.  

Now we can apply automorphy lifting theorems, as in \cite{BGGT}, to show that $r_7 \circ \tilde{\rho}$ is potentially automorphic -- that its restrictions to appropriate totally real Galois extensions $F'/F$ are attached to a cuspidal cohomological self-dual automorphic representation $\pi'$ of $GL(7,\ad_{F'})$.  Choosing $F'$ carefully, we can then descend $\pi'$ to an automorphic representation $\pi''$ of $GL(7,\ad_{F''})$ over the fixed field $F''$ of a decomposition group $\operatorname{Gal}(F'_{v'} / F_v) \subset \operatorname{Gal}(F'/F)$.  At this point we apply the result of Hundley and Liu to show that $\pi''$ is in the image of the functorial transfer from $G_2(\ad_{F''})$ to $GL(7,\ad_{F''})$ of an automorphic representation $\Pi$ of $G_2(\ad_{F''})$, and we conclude by observing that the local component $\Pi_v$ is supercuspidal and has parameter $\rho$.  As a bonus, the construction of \cite{HL} provides a globally generic $\Pi$, so we see that $\rho$ is the parameter of a generic supercuspidal representation.

There has been a good deal of work on the local representation theory of as well as the automorphic theory of $G_2$.  Notably, the articles \cite{GS04}, \cite{SW07}, and \cite{SWe} come very close to establishing a complete local Langlands correspondence for $G_2$ and to relate the correspondence to the exceptional theta correspondence used here\footnote{While revising this article we learned of the new preprint \cite{GS21} of Gan and Savin that establishes Howe duality as well as
a dichotomy result for these exceptional theta correspondences.   It is likely that some of the arguments in the present paper can now be simplified  and many of the references can be consolidated, but we have not had time to study the new paper.}
; the article \cite{HL} comes very close to characterizing the image of functoriality from $G_2$ to $GL(7)$.  The purpose of this article is 
not to replace the articles just cited -- indeed, the results of these articles are used crucially in the proof of our main theorem -- but rather to illustrate the possibility of 
applying a combination of arithmetic and automorphic methods to the local correspondence.

\subsection*{Acknowledgements}

As mentioned above, this paper implements a strategy that was developed in our joint paper with Gebhard B\"ockle, and we are grateful to him for many discussions.  Thanks are due to Joseph Hundley and Baiying Liu for bringing to our attention their recent result on descent for $G_2$, on which our argument crucially depends.   We also thank Jeff Adams, Wee Teck Gan, Dihua Jiang, Aaron Pollack, and Gordan Savin for help with references, and Freydoon Shahidi for pointing out
an error in the $L$-function calculation in the original proof of Proposition \ref{prop_global_data_implies_injectivity}.  We  thank Savin for agreeing to write the appendix that proves a global genericity result that allows us to define the local parameterization unambiguously.  

We also thank the anonymous referees for their very careful reading, and for their many suggestions that, we hope, have made the text more readable.

M.H. was partially supported by NSF Grant DMS-1701651.    C.K. was partially supported by NSF Grant DMS-1601692. J.T.'s work received funding from the European Research Council (ERC) under the European Union's Horizon 2020 research and innovation programme (grant agreement No 714405). 

\subsection*{Notation}

If $K$ is a perfect field, we will write $\Gamma_K$ for its Galois group relative to a fixed algebraic closure. When $K$ is a number field, we will fix an algebraic closure $\overline{K} / K$, algebraic closures $\overline{K}_v / K_v$ for each place $v$ of $K$, and embeddings $\overline{K} \to \overline{K}_v$ extending $K \to K_v$. These choices determine embeddings $\Gamma_{K_v} \to \Gamma_K$ for each place $v$ of $K$. If $v$ is a finite place, then $I_{K_v} \subset \Gamma_{K_v}$ is the inertia group.

We write $\epsilon : \Gamma_K \to \ZZ_\ell^\times$ for the $\ell$-adic cyclotomic character. By abuse of notation, we also write $\epsilon$ for the pushforward of this character to the group of units of any $\ZZ_\ell$-algebra.

If $F$ is a totally real number field, $n$ is an \emph{odd} integer, and $\pi$ is a cuspidal, regular algebraic automorphic representation of $GL(n, \mathbb{A}_F)$ which is self-dual, in the sense that $\pi \cong \pi^\vee$, then for any isomorphism $\iota : \overline{\QQ}_\ell \to \CC$ there is an associated semi-simple $\ell$-adic Galois representation $r_\iota(\pi) : \Gamma_F \to GL(n, \overline{\QQ}_\ell)$. This is characterized, up to isomorphism, by its compatibilty with the local Langlands correspondence for $GL(n)$ at finite places. More precisely, if $v \nmid \ell$ is any finite place of $F$ then there is an isomorphism
\[ \operatorname{WD}(r_\iota(\pi)|_{\Gamma_{F_v}})^{F-ss} \overset{\sim}{\to} \iota^{-1} \operatorname{rec}_{F_v}( \pi_v). \]
(The notations here are as defined in \cite[pp. 509--510]{BGGT}: $\operatorname{WD}$ denotes the Weil--Deligne representation associated to an $\ell$-adic representation, $F-ss$ denotes the Frobenius-semisimplification of a Weil--Deligne representation, and $\operatorname{rec}_{F_v}$ is the local Langlands correspondence for the group $GL(n, F_v)$.) 

There is an isomorphism $r_\iota(\pi)^\vee \cong r_\iota(\pi)$. We note that our representation $r_\iota(\pi)$ differs by a Tate twist from the representation $r_{l, \iota}(\pi)$ defined in \cite[Theorem 2.1.1]{BGGT}. Our normalization, which only makes sense when $n$ is odd, suits our purposes here since we want representations $\pi$ which arise as functorial lifts from $G_2$ to give rise to Galois representations which are pure of weight 0.

\section{Galois parameterization of $G_2$}\label{sec_galois_par}

Let $G_2$ be the split group of that type over $\ZZ$. Let $K$ be a local field of characteristic 0 and let $\CA_g(G_2,K)$ denote the set of equivalence classes of generic irreducible admissible  representations of $G_2(K)$ over $\CC$.  If $K$ is non-archimedean, we let $\CA^0_g(G_2,K) \subset \CA_g(G_2,K)$ denote the subset of {\it supercuspidal} representations.   Let  $\CG(G_2,K)$ denote the set of $G_2(\CC)$-conjugacy classes of $G_2$-completely reducible parameters 
$$\rho : W_K \ra G_2(\CC),$$ and let $\CG^0(G_2,K) \subset \CG(G_2,K)$ denote the subset of classes of $G_2$-irreducible parameters. We define $\CG(GL(7), K)$ and $\CG^0(GL(7), K)$ similarly.

The aim of \S \ref{sec_galois_par} is to collect lifting results scattered in the literature (among them \cite{GS04, GJ, GRS97, GS98, GW, HPS, HL, Li99, MS, SWe, SW15}) to construct a map
\begin{equation}\label{param}  \CL_g:  \CA^0_g(G_2,K) \rightarrow \CG^0(G_2,K).
\end{equation}
We begin by constructing a map 
\begin{equation}\CL_g' : \CA_g(G_2, K) \to \CG(GL(7), K)
\end{equation}
using purely local means. Conjugacy results for $G_2$ (see \cite{Gri95}) imply that the map 
\begin{equation} r_{7, \ast} : \CG(G_2, K) \to \CG(GL(7), K)
\end{equation}
determined by the standard representation $r_7$ of $G_2$ is injective, so the main problem is to show that that $\CL'_g(\CA_g^0(G_2, K))$ lies in the image of $r_{7, \ast}$. This we achieve using a global argument.

We note that if $K$ is a local field of positive characteristic, Genestier and Lafforgue construct the analogue of the map $\CL_g'$ in \cite{GLa} without reference to the theta lift, and at the same time show that its image is contained in the image of $r_{7, \ast}$.

\subsection{Local generic theta lift}\label{thetaR}

Let $K$ be a local field of characteristic 0, and let $E_i$ $(i = 6, 7)$ denote the split adjoint reductive group of that type. We set $H_i = PGL(3)$ (if $i = 6$) and $H_i = PSp(6)$ (if $i = 7$). Let $\theta_{i,K}$ be the minimal representation of $E_i(K)$, which we consider by restriction
to be a representation of $G_2(K)\times H_i(K)$, $i = 6, 7$.  When $K$ is archimedean we work with the Harish-Chandra module of $\theta_{i,K}$ relative to a choice of maximal compact subgroup of $E_i(K)$ that contains a chosen product of maximal compact subgroups of $G_2(K)$ and $H_i(K)$.  
Let $\pi$ be an irreducible admissible representation of $G_2(K)$, and let
$\theta_{i,K,[\pi]}$ denote the maximal quotient of $\theta_{i,K}$ which is isotypic for $\pi$ as representation of $G_2(K)$.  Then we write
\begin{equation}\label{thetalocal}
\theta_{i,K,[\pi]} = \pi \otimes \Theta_i(\pi)
\end{equation}
where $\Theta_i(\pi)$ is a smooth representation of $H_i(K)$.  
\begin{prop}\label{prop_uniqueness_of_local_theta_lift}
	\begin{enumerate}
		\item 	Suppose that $K$ is non-archimedean and that $\pi$ is generic. Then $\Theta_7(\pi)$ admits a unique generic subquotient.
		\item Suppose that $K = \RR$ and $\pi$ is a discrete series representation. If $\Theta_7(\pi)$ admits a generic constituent, then it is uniquely determined by $\pi$, up to isomorphism.
	\end{enumerate}
\end{prop}
\begin{proof}
	The first part is \cite[Corollary 20]{GS04}.  This is based on Proposition 19 of \cite{GS04}, whose proof there is sketched.  The proposition is restated, with a complete proof, as Theorem \ref{thm:main-appendix} of Savin's appendix.   The second follows from the fact that $\Theta_7(\pi)$ has regular integral infinitesimal character determined by that of $\pi$, hence any generic constituent of $\Theta_7(\pi)$ must be the unique discrete series representation with that infinitesimal character. We describe this in more detail below.
\end{proof}
When $\Theta_7(\pi)$ has a unique generic subquotient, we denote it by $\theta_7(\pi)$. We can specify $\theta_7(\pi)$ precisely in two important special cases. First, the unramified, non-archimedean case:
\begin{proposition}[{\cite[Theorem 3.5]{GJ}, \cite[Theorem 1.1]{SW07}}]\label{unramG2}  Suppose that $K$ is non-archimedean and that $\pi$ is an unramified, generic representation of $G_2(K)$. Then $\theta_7(\pi)$ is the unramified representation of $PSp(6, K)$ determined by the embedding $G_2 \to Spin(7)$ of $L$-groups. 
\end{proposition}

\begin{proof} The article \cite{GJ} determines the Satake parameter of the theta lift $\Theta'(\pi)$ of $\pi$ to $GSp(6,K)$ by comparing the local Euler factors for the $8$-dimensional Spin representation of the Langlands dual group $GSpin(7)$ of $GSp(6)$ (see the displayed formula at the bottom of p.\ 42 of \cite{GJ}).  The genericity of 
	$\Theta'(\pi)$ is proved in the course of this comparison.  The representation $\Theta'(\pi)$ is pulled back from the representation $\Theta_7(\pi)$ of $PSp(6,K)$, which implies that the Satake parameter of $\Theta'(\pi)$ lies in the subgroup $Spin(7)$ of $GSpin(7)$ and has the indicated form.  See also the proof of \cite[Proposition 5.2]{KLS10}.
\end{proof}
Next, the real, discrete series case. We first recall an important fact. Let $G$ be a quasi-split connected reductive group over $\RR$ such that $G(\RR)$ admits discrete series representations, and let $K \subset G(\RR)$ denote a maximal compact subgroup.  Let $B \subset G(\RR)$ be a Borel subgroup with unipotent radical $N$.  Let $Z(\fg)$ denote the center of the enveloping algebra $U(\fg)$ of the complexified Lie algebra of $G$.   If $\pi$ is an irreducible representation of $G(\RR)$, we let $\xi_\pi:  Z(\fg) \ra \CC$ denote its infinitesimal character.   We recall
\begin{fact}\label{inf}  If $W$ is an irreducible (algebraic) representation of $G$, then, up to infinitesimal equivalence, there is a bijection between
Weyl chambers $C$ for $G$, modulo the Weyl group of $K$, such that all simple roots in the chamber are non-compact,
and unitary representations $\pi(C,W)$ of $G(\RR)$ such that 
\begin{itemize}
\item $\xi_W = \xi_{\pi(C,W)}$. 
\item $\pi(C,W)$ is generic.
\end{itemize}
Moreover, $\pi(C,W)$ is discrete series. 
\end{fact}
The experts tell us that the literature contains no explicit statement of this well-known (and frequently cited) fact.  
At the request of the referee we provide a sketch.   We thank Gordan Savin and Jeff Adams for pointing out
errors in our first attempt and providing some missing references.
 
Salamanca-Riba proved in \cite{SR} that a unitary representation with infinitesimal character $\xi_W$
is necessarily cohomological (an $A_{\mathfrak{q}}(\lambda)$).  Now Theorem K of \cite{kos} asserts that a generic representation
is ``large'' in the sense of \cite{V}, and it thus follows from Theorem 6.2 of \cite{V} that a generic unitary representation of $G(\RR)$ can never be
a Langlands quotient of a reducible principal series representation.  But Theorem 6.16 of \cite{VZ} implies that an 
$A_{\mathfrak{q}}(\lambda)$ is a (proper) Langlands quotient unless $\mathfrak{q}$ is a Borel subgroup.  Thus
a generic $A_{\mathfrak{q}}(\lambda)$ is  in the discrete series, by our hypothesis on $G$.   Then the bijection follows from \cite[Lemma 5.7]{Ad}, which is a restatement in more contemporary language of \cite[Theorem 6.2]{V} as applied to discrete series.

One checks that when $G = G_2(\RR)$ or $G = PSp(6,\RR)$,  there is a unique Weyl chamber $C$ as in Fact \ref{inf}; thus we 
may write $\pi(W)$ instead of $\pi(C,W)$.   We follow the discussion
in \cite{Li97} of the Conjecture of Gross, which identifies the Harish-Chandra parameters of the generic discrete series of both groups explicitly.  
We  fix normalizations, first for $G_2$, and then for $PSp(6)$. Let $\omega_1$ denote the highest weight of the irreducible $7$-dimensional representation of $G_2$, and let $\omega_2$ be the other fundamental weight. Given non-negative integers $a, b$ we let $W(a,b)$ denote the irreducible representation of $G_2$ with highest weight $a\omega_1 + b\omega_2$, and write $\pi_{a,b}$ for $\pi(W(a,b))$, a discrete series representation of $G_2(\RR)$.  It follows
from Fact \ref{inf} that the generic representation for $G_2(\RR)$ with given infinitesimal character is unique; it is the one with the marking (3) on p. 191 of
\cite{Li97}.   For sufficiently regular infinitesimal character, $\pi_{a,b}$ belongs to the {\bf integrable} discrete series \cite{HS}.  Theorem \ref{arch} below is a partial confirmation of Gross's conjecture for generic discrete series.

We denote characters $Z(\mathfrak{sp}(6)) \ra \CC$ in the standard way by triples $(\alpha,\beta,\gamma)$ of integers with $\alpha > \beta > \gamma > 0$ (cf \cite{Li97}), and let 
$W(\alpha,\beta,\gamma)$ be the irreducible algebraic representation of $PSp(6)$ with the corresponding infinitesimal character.   

\begin{thm}\label{arch} 
	Let $a, b$ be non-negative integers; we assume $\pi_{a,b}$ to be in the integrable discrete series.  Then 
	$$\theta_7(\pi_{a,b}) = \pi(W(a+2b+3,a+b+2,b+1)).$$
\end{thm}
\begin{proof}
	See \cite[Theorem 5.4]{HPS}. In more detail, let us use the superscript $(?)^c$ to denote the compact form of a real reductive group.  The article \cite{HPS} treats the exceptional theta correspondence $(G_2(\RR),PSp(6, \RR)^c)$ (among others),  \cite{GW} treats (among others) the correspondence $(G_2(\RR)^c,PSp(6,\RR))$, and \cite{Li97} treats the  split case $(G_2(\RR),PSp(6,\RR))$ but only computes the correspondence for quaternionic discrete series of $G_2(\RR)$. However, as observed in \cite[p.\ 204]{Li97}, the correspondence of infinitesimal characters is independent of real forms.  The determination of the correspondence for infinitesimal characters of generic discrete series is completed in \cite{Li99} (see Table 1 on p.\ 375 of that paper).
	
	By Fact \ref{inf} this suffices to identify $\theta_7(\pi_{a, b})$, once we know that it is {\it unitary}.   This is where we use the hypothesis	that $\pi_{a,b}$ is in  the integrable discrete series.  Under that hypothesis, we can construct a globally generic cuspidal automorphic 
	representation $\Pi$ of $G_2(\QQ)$ with local component $\pi_{a,b}$ at the real place, as in the proof of Proposition \ref{prop_image_in_G_2} below.
	Theorem \ref{globalGRS} then asserts that the global theta lift $\Theta_7(\Pi)$	of $\Pi$ to an automorphic representation of $PSp(6,\QQ)$ does not vanish and is cuspidal.   In particular, the archimedean component $\Theta_7(\Pi)_\infty$ is generic and unitary.  By Proposition \ref{prop_uniqueness_of_local_theta_lift} and the
	compatibility of local and global correspondences, $\Theta_7(\Pi)_\infty = \theta_7(\pi_{a,b})$; this completes the proof.	
\end{proof}
\begin{subremark}  Gross's conjecture does not assume that $\pi_{a,b}$ is in the integrable discrete series 
and the hypothesis is certainly unnecessary.  The argument above applies
whenever $\pi_{a,b}$ can be realized as a local constituent of a globally generic cuspidal automorphic
representation.  Undoubtedly this is always possible, but the construction we use here requires the archimedean component to be in
the integrable discrete series.
\end{subremark}

\subsection{Global generic theta lift}

Now let $F$ be a totally real number field.  When $G = E_i$, we let $\CA(G)$ denote the space of automorphic forms on $G(F)\backslash G(\ad_F)$, and we let $\theta_i := \theta_G \subset \CA(G)$ denote the minimal automorphic representation, as described in the article \cite{GJ} (see also \cite{GRS97}). Let $\pi$ be a cuspidal automorphic representation of $G_{2}(\ad_F)$.  We define $\Theta_7(\pi)$ and $\Theta_6(\pi)$ to be the spaces of automorphic forms on $H_7 = PSp(6, \ad_F)$ and $H_6 = PGL(3, \ad_F)$, respectively, defined to be the span of the functions
$\Theta_i(f_\theta,\varphi)$, as $f_\theta \in \theta_G$ and $\varphi$ runs through the automorphic forms in the space of the {\it contragredient} $\pi^{\vee}$ of $\pi$, and where
\begin{equation}\label{thetaglobal}
\Theta_i(f_\theta,\varphi)(h) = \int_{[G_2]} f_\theta(g,h)\varphi(g) dg,\,\, h \in H_i(\ad_F).
\end{equation}
Here the notation $\int_{[G_2]}$ is the standard abbreviation of $\int_{G_2(F)\backslash G_2(\ad)}$, and $(g,h)$ is a variable element of 
$G_2(\ad_F) \times H_i(\ad_F) \subset E_i(\ad_F)$.  In contrast to \cite{GRS97}, we let $\varphi$ to belong to $\pi^{\vee}$ (or equivalently to the complex conjugate of $\pi$) to guarantee compatibility with the local correspondence defined below.

\begin{thm}\label{globalGRS} 
	Let $\pi$ be a cuspidal automorphic representation of $G_2(\ad_F)$. Then:
	\begin{enumerate}
		\item Let $\Pi$ be an irreducible subquotient of $\Theta_i(\pi)$. Then for each place $v$ of $F$, $\Pi_v$ is an irreducible subquotient of $\Theta_i(\pi_v)$. 
		\item Suppose that $\pi$ is globally generic and that $\pi_\infty$ is a discrete series representation. Then $\Theta_6(\pi) = 0$, and $\Theta_7(\pi)$ is cuspidal and globally generic. In particular, it is non-zero. 
	\end{enumerate}
\end{thm}
\begin{proof}
	The first part is a formal consequence of the definition. The second part is a consequence of \cite[Theorem B]{GRS97} (which shows that $\Theta_7(\pi)$ is non-zero), the vanishing of $\Theta_6(\pi)$ (which follows from the fact that $\pi_\infty$ is discrete series, using the same argument as in \cite[\S 6]{KLS10}), and Theorem 1.4 of Savin's appendix to this paper (which shows that $\Theta_7(\pi)$ is globally generic).
\end{proof}
We will only apply Theorem \ref{globalGRS} in the case that $\pi$ is a cuspidal, globally generic representation of $G_2(\ad_F)$ such that $\pi_\infty$ is discrete series. In this case, $\Theta_7(\pi)$ is a cuspidal, globally generic representation of $PSp(6, \ad_F)$, and we write $\theta_7(\pi)$ for the unique globally generic cuspidal irreducible constituent of its restriction to $Sp(6,F)\backslash Sp(6,\ad_F)$. We observe that for every place $v$ of $F$, we have the equality $\theta_7(\pi_v) = \theta_7(\pi)_v$ (compatibility with local generic theta correspondence).

\subsection{Lifting to $GL(7)$}

There are at least two ways of lifting from $Sp(6)$ to $GL(7)$ (see \cite{CKPS, A}). We use the lifting constructed in  \cite{CKPS}. 
\begin{thm}\label{thm_Arthur}
	Let $K$ be a local field of characteristic 0. Then there exists a map $\Psi$ from the set of (isomorphism classes of) generic irreducible admissible representations of $Sp(6, K)$ to the set of (isomorphism classes of) generic irreducible admissible representations of $GL(7, K)$. The map $\Psi$ has the following properties:
	\begin{enumerate}
		\item If $K$ is non-archimedean then $\Psi$ is injective.
		\item If $K$ is non-archimedean and $\pi$ is an unramified representation of $Sp(6, K)$, then $\Psi(\pi)$ is the unramified representation of $GL(7, K)$ determined by the natural embedding $SO(7) \to GL(7)$ of $L$-groups.
		\item If $K = \RR$ and $\pi$ is the generic discrete series representation of $Sp(6, K)$ with infinitesimal character $(\alpha, \beta, \gamma)$, then $\Psi(\pi)$ is the unique generic representation of $GL(7, \RR)$ which is cohomological for the irreducible algebraic representation of $GL(7, \RR)$ of highest weight $(\alpha-3, \beta-2, \gamma-1, 0, 1-\gamma, 2-\beta, 3-\alpha)$.
		\item Let $F$ be a number field, and let $\Pi$ be a globally generic cuspidal automorphic representation of $Sp(6, \ad_F)$. Then there exist self-dual, cuspidal automorphic representations $\Psi_1, \dots, \Psi_r$ of $GL(n_i, \ad_F)$, where $\sum_{i=1}^r n_i = 7$, with the following property: let $\Psi = \Psi_1 \boxplus \dots \boxplus \Psi_r$. Then for any place $v$ of $F$, $\Psi_v = \Psi(\Pi_v)$. 
	\end{enumerate}
\end{thm}
\begin{proof}
 The paper \cite{CKPS} defines the local functorial lift of a generic irreducible admissible representation $\pi$ of $Sp(6, K)$ by agreement of $L$-functions and $\epsilon$-factors of pairs, see \cite[Definition 7.1]{CKPS}. This lift is uniquely characterized by the local converse theorem for $\GL(7, K)$, and 	\cite[Proposition 7.5]{CKPS} asserts that it always exists. This defines the map $\Psi$. The local converse theorem for $p$-adic $Sp(6, K)$, proved in \cite{Zh}, shows that $\Psi$ is injective in this case. The final part of the theorem follows from \cite[Theorem 7.2, Proposition 7.2]{CKPS}.
\end{proof}

%
%
%
\subsection{Definition of $\CL_g'$ and $\CL_g$}

Let $K$ be a non-archimedean local field of characteristic 0. We can now define the promised map $\CL_g' :  \CA_g(G_2, K) \to \CG(GL(7), K)$. If $\pi \in \CA_g(G_2, K)$, then we define $\CL_g'(\pi) = \operatorname{rec}_K( \Psi(\sigma) )^{ss}$, where $\sigma$ is the unique irreducible constituent of $\theta_7(\pi)|_{Sp(6, K)}$ which is generic (with respect to our fixed choice of Whittaker datum). This is independent of the choice of Whittaker datum on $Sp(6)$. The effect of semisimplification here is to remove the nilpotent part of the Weil--Deligne representation  $\operatorname{rec}_K( \Psi(\sigma) )$, necessary because of our convention that $\CG(GL(7), K)$ consists of equivalence classes of semisimple representations of $W_K$.

The next step is to show that when $\pi$ is supercuspidal, $\CL_g'(\pi)$ can be conjugated to take values in $G_2(\CC)$. We will establish this using a global argument.
\begin{prop}\label{prop_image_in_G_2}
	Let $\pi \in \CA^0_g(G_2, K)$. Then $\CL_g'(\pi)$ lies in the image of $r_{7, \ast}$.
\end{prop}
\begin{proof}
	Fix a totally real number field $F$ with a finite place $v$ such that $F_v \isoarrow K$.  Let $\alpha_1, \dots, \alpha_r$ be the real places of $F$.  For each $i = 1, \dots, r$, we choose a generic integrable discrete series representation $\pi_i$ of the (split) group $G_2(\RR)$. Fix a finite place $w \neq v$ of $F$ of residue characteristic different from 2 and a generic supercuspidal representation $\pi_w$ of $G_2(F_w)$ with the property that $\theta_7(\pi_w)$ is a supercuspidal representation of $PSp(6,F_w)$. (Such representations are constructed explicitly in the proof of Theorem 5.2 of \cite{KLS10}, at least when $F_w = \QQ_p$, which we assume for convenience of reference.) By \cite[Theorem 4.5]{KLS08} (see also \cite[Theorem 2.2]{V84}), we can find a cuspidal, globally generic automorphic representation $\Pi$ of $G_2(\ad_F)$ such that $\Pi_v \cong \pi$, $\Pi_w \cong \pi_w$, and $\Pi_{\alpha_i} \cong \pi_i$ for each $i$. 
	
	Then $\theta_7(\Pi)$ is globally generic and cohomological, and its local component at $w$ is supercuspidal, by construction. Applying Theorem \ref{thm_Arthur} to the restriction of $\theta_7(\Pi)$ to $Sp(6, \ad_F)$, we can find cuspidal, self-dual automorphic representations $\Psi_1, \dots, \Psi_r$ of $GL(7, \ad_F)$ such that $\Psi(\theta_7(\Pi)) = \Psi_1 \boxplus \dots \boxplus \Psi_r$ is cohomological. In particular, each $\Psi_i$ is cohomological up to twist. Fixing an isomorphism $\iota : \overline{\QQ}_\ell \to \CC$, we get a compatible family $r_\iota(\Psi_i)$ of $n_i$-dimensional representations with values in $GL(n_i,\Qlb)$.  We get a representation
	\begin{equation}\label{GLrep}  r_\iota(\Pi) := r_\iota(\Psi(\theta_7(\Pi))) : \Gamma_F \to GL(7, \overline{\QQ}_\ell).
	\end{equation}
	In fact, \cite[Theorem 6.4]{Ch} shows that the representation $r_\iota(\Pi)$ is conjugate to a representation contained in $r_7(G_2(\overline{\QQ}_\ell))$, and such a representation is unique up to conjugation in $G_2(\overline{\QQ}_\ell)$ (noting that the cuspidality of $\Psi$ plays no role in the proof of that theorem). Our proof is now complete: we have $r_\iota(\Pi)|_{W_{F_v}}^{ss} \cong \iota^{-1} \operatorname{rec}_{F_v}(\Psi(\theta_7(\pi)))^{ss} \cong \iota^{-1} \CL_g'(\pi)$, by definition, and this shows that $\CL_g'(\pi)$ is conjugate to a representation valued in $G_2(\CC)$.
\end{proof}
\begin{prop}\label{prop_irreducibility}
	Let $\pi \in \CA^0_g(G_2,K)$, and let $\rho : W_K \to G_2(\CC)$ be the unique parameter, up to $G_2(\CC)$-conjugacy, with $r_7 \circ \rho = \CL_g'(\pi)$. Then $\rho$ is irreducible. 
\end{prop}
\begin{proof}
	We split up into cases according to the dichotomy described in \cite{SWe}. If $\theta_7(\pi)$ is supercuspidal then by \cite[Theorem  7.2]{CKPS} there are self-dual, non-isomorphic supercuspidal representations $\psi_1, \dots, \psi_r$ of $GL(n_i, K)$ with $\sum_{i=1}^r n_i = 7$ and $r_7 \circ \rho$ conjugate in $GL(7, \CC)$ to $\oplus_{i=1}^r \operatorname{rec}_K(\psi_i)$. There is a unique way to conjugate such a parameter into $SO(7, \CC)$, and it is irreducible there (indeed, its centralizer is finite). It follows that $\rho$ must also be irreducible in this case.
	
	If $\theta_7(\pi)$ is not supercuspidal, then \cite[Proposition 3.6]{SWe} shows that $\theta_7(\pi)$ is a subquotient of an unnormalized induction $\operatorname{Ind}_{Q_3}^{PSp(6)} \rho \otimes | \det |$, where $\rho$ is a supercuspidal representation of $PGL(3, K)$. By \cite[Proposition 7.4]{CKPS}, $r_7 \circ \rho$ is conjugate in $GL(7, \CC)$ to $\operatorname{rec}_K(\rho) \oplus \CC \oplus \operatorname{rec}_K(\rho)^\vee$. Since $\operatorname{rec}_K(\rho)$ is irreducible and 3-dimensional, \cite[Proposition 1.5]{SWe} shows that the image of this representation is not contained in the image of any Levi subgroup of $G_2$, so once again $\rho$ must be irreducible.
\end{proof}
We may therefore complete our definition of $\CL_g$ as follows: if $\pi \in  \CA^0_g(G_2,K)$ then $\CL_g(\pi)$ is the unique parameter, up to conjugacy, with $r_7 \circ \CL_g(\pi) = \CL_g'(\pi) $. We record the following useful property of $\CL_g$.
\begin{prop}\label{prop_supercusp}
	Let $\pi \in \CA^0_g(G_2,K)$. and let $R : G_2 \to SO(7)$ be the unique non-trivial homomorphism. Then $\theta_7(\pi)$ is supercuspidal if and only if $R \circ \CL_g(\pi)$ is $SO(7,\CC)$-irreducible.
\end{prop}
\begin{proof}
	This is a corollary of the proof of Proposition \ref{prop_irreducibility}.
\end{proof}

\section{Globalization of local $G_2$ parameters}\label{GaloisLiftings}

Let $G = G_2$ be a split reductive group over $\ZZ$ of that type, let $\mathfrak{g}$ denote its Lie algebra, and let $r_7 : G \to GL(7)$ denote the standard 7-dimensional representation. We fix a split maximal torus and Borel subgroup $T \subset B \subset G$. We may assume that $r_7(T)$ is diagonal and $r_7(B)$ is contained in the upper-triangular Borel subgroup of $GL(7)$. Let $\Delta \subset \Phi = \Phi(G, T)$ denote the corresponding root basis, and $\Phi = \Phi^+ \sqcup \Phi^-$ the sets of positive and negative roots. We label $\Delta = \{ \alpha_1, \alpha_2 \}$ so that the fundamental weight $\omega_1$ is the highest weight of $r_7$. Let $\check{\omega}_1, \check{\omega}_2 \in X_\ast(T)$ denote the corresponding fundamental coweights, and let $\check{\delta} = \check{\omega}_1 + \check{\omega}_2$.

Let $K$ be finite extension of $\QQ_p$, and let $k$ be a finite field of characteristic $\ell \neq p$. We suppose given a continuous representation $\overline{\rho} : \Gamma_K \to G(k)$ such that $\overline{\rho}(I_K)$ has order prime to $\ell$.
\begin{thm}\label{approx}
	There exists a totally real field $F$ and a continuous representation $\overline{\sigma} : \Gamma_F \to G(k)$ with the following properties:
	\begin{enumerate}
		\item $\overline{\sigma}(\Gamma_F) = G(k)$.
		\item $\ell$ splits in $F$. For each place $v | \ell$ of $F$, $\overline{\sigma}|_{\Gamma_{F_v}}$ is $G(k)$-conjugate to $\check{\delta}^2 \circ \epsilon$.
		\item For each place $v | p$ of $F$, there is an isomorphism $F_v \cong K$ such that $\overline{\sigma}|_{\Gamma_{F_v}}$ is $G(k)$-conjugate to $\overline{\rho}$.
		\item If $v$ is a finite place of $F$ such that $v \nmid \ell p$, then $\overline{\sigma}|_{\Gamma_{F_v}}$ is unramified.
		\item For each place $v | \infty$ of $F$ \textup{(}which determines a complex conjugation $c_v \in \Gamma_{F_v}$\textup{)}, $\overline{\sigma}(c_v)$ is $G(k)$-conjugate to $\check{\delta}(-1)$.
	\end{enumerate}
\end{thm}
\begin{proof}
	By weak approximation, we can find a totally real number field $E / \QQ$ in which $\ell$ splits, and such that for each $v | p$ there is an isomorphism $E_v \cong K$. The existence of an extension $F / E$ and a homomorphism $\overline{\sigma} : \Gamma_F \to G(k)$ with the claimed properties then follows from the main theorem of \cite{M90}.
\end{proof}
\begin{thm}\label{lifting}
	Let $\overline{\sigma} : \Gamma_F \to G(k)$ be a representation satisfying the conclusion of Theorem \ref{approx}, and suppose that $\ell > 28$. Then there exists a finite extension $E / \operatorname{Frac} W(k)$ with ring of integers $\cO$ and a lift
	\[ \sigma : \Gamma_F \to G(\cO) \]
	of $\overline{\sigma}$ with the following properties:
	\begin{enumerate}
		\item $\sigma(\Gamma_F)$ contains a conjugate of $G(\ZZ_\ell)$.
		\item For each place $v | \ell$ of $F$, there exists $g \in G(\cO)$ such that $g \sigma|_{\Gamma_{F_v}} g^{-1}$ takes values in $B(\cO)$, and the projection of $g \sigma|_{I_{F_v}} g^{-1}$ to $T(\cO)$ equals $\check{\delta}^2 \circ \epsilon$. Consequently, $r_7 \circ \sigma|_{\Gamma_{F_v}}$ is crystalline ordinary of Hodge--Tate weights $\{ 6, 4, 2, 0, -2, -4, -6 \}$ (with respect to any embedding $F_v \hookrightarrow \overline{\QQ}_\ell$).
		\item  For each place $v | p$ of $F$, reduction modulo $\mathfrak{m}_\mathcal{O}$ induces an isomorphism $\sigma(I_{F_{v}}) \overset{\sim}{\to} \overline{\sigma}(I_{F_{v}})$.
		\item For each finite place $v \nmid \ell p$ of $F$, $\sigma|_{\Gamma_{F_v}}$ is unramified.
		\item For each place $v | \infty$ of $F$, $\sigma(c_v)$ is $G(\cO)$-conjugate to $\check{\delta}(-1)$.
	\end{enumerate}
\end{thm}
\begin{proof} 
	We first remark that our hypotheses imply that $G(k)$ is its own derived group, hence $\overline{\sigma}(\Gamma_{F(\zeta_\ell)}) = G(k)$, and $r_7 \circ \overline{\sigma}$ is absolutely irreducible (see \cite{St}). This will be useful in applying the results of \cite{BGGT} cited below. We next observe that \cite[Proposition 6.7]{BHKT} shows that any lift $\sigma$ of $\overline{\sigma}$ to $G(\cO)$ has the property that $\sigma(\Gamma_F)$ contains a conjugate of $G(\ZZ_\ell)$. 
	
	To construct a lift, we use the Khare--Wintenberger method (cf. \cite[Theorem 3.7]{KW}). Let $\mathcal{C}$ denote the category of complete local Noetherian $W(k)$-algebras with residue field $k$, and let $\operatorname{Def}_G : \mathcal{C} \to \operatorname{Sets}$ denote the functor which assigns to any $A \in \mathcal{C}$ the set of $\ker(G(A) \to G(k))$-conjugacy classes of homomorphisms $\sigma_A : \Gamma_F \to G(A)$ satisfying the following conditions:
	\begin{itemize}
		\item For each place $v | \ell$ of $F$, there exists $g \in G(A)$ such that $g \sigma_A|_{\Gamma_{F_v}} g^{-1}$ takes values in $B(A)$, and the projection of $g \sigma_A|_{I_{F_v}} g^{-1}$ to $T(A)$ equals $\check{\delta}^2 \circ \epsilon$.
		\item For each place $v | p$ of $F$, reduction modulo $\mathfrak{m}_A$ induces an isomorphism $\sigma_A(I_{F_{v}}) \overset{\sim}{\to} \overline{\sigma}(I_{F_{v}})$.
		\item For each finite place $v \nmid \ell p$ of $F$, $\sigma_A|_{\Gamma_{F_v}}$ is unramified.
	\end{itemize}
	Then (\cite[Proposition 9.2]{Pat16}) $\operatorname{Def}_G$ is represented by an object $R_G \in \mathcal{C}$, and there exists an integer $g \geq 0$ such that $R_G$ can be expressed as a quotient of $W(k)[[X_1, \dots, X_g]]$ by $g$ relations. (Note that the local conditions are liftable local deformation conditions, in the sense of \cite{Pat16}, which are analyzed in \cite[\S 4.1]{Pat16} and \cite[\S 4.4]{Pat16}. The conditions $\ell > 28$ and $\ell$ split in $F$ imply that conditions (REG) and (REG*) of  \cite[\S 4.1]{Pat16} hold.) To apply the Khare--Wintenberger method, we must show that $R_G$ is a finite $W(k)$-algebra. This will imply that $R_G$ is a finite flat complete intersection $W(k)$-algebra, and therefore that there exists a finite extension $W(k) \to \cO$ and a homomorphism $R_G \to \cO$, corresponding to a lift $\sigma : \Gamma_F \to G(\cO)$ with the desired properties.
	
	To prove the finiteness of $R_G$, we must compare it with another deformation ring. Let $E / F$ be a totally imaginary quadratic extension in which the places of $F$ above $\ell$ and $p$ split; then $E$ is a CM field. Let $\overline{\tau} = r_7 \circ \overline{\sigma}|_{\Gamma_E}$. Then $\overline{\tau}$ is absolutely irreducible and $\overline{\tau}^c \cong \overline{\tau}^\vee$ (where $c \in \operatorname{Gal}(E / F)$ is the non-trivial element). Let $\mathcal{G}_7 = (GL(7) \times GL(1)) \rtimes \{ \pm 1 \}$ denote the group scheme defined in \cite[\S 1.1]{BGGT}. It is equipped with a homomorphism $\nu : \mathcal{G}_7 \to GL(1)$ which is projection to $GL(1)$ on the connected component and which sends $-1$ to $-1$. As described there, $\overline{\tau}$ extends to a homomorphism $\overline{\tau} : \Gamma_F \to \mathcal{G}_7(k)$ such that $\nu \circ \overline{\tau} = \delta_{E/F}$ and $\overline{\tau}(\Gamma_E) \subset GL(7, k)$. (We write $\delta_{E/F} : \operatorname{Gal}(E/F) \to \{ \pm 1 \}$ for the unique non-trivial character.)
	
	Let $S$ denote the set of places of $F$ dividing $\ell p$. Fix for each $v \in S$ a place $\wv$ of $E$ lying above $v$. Say that a finite totally real extension $F' / F$ is good if $\overline{\tau}(\Gamma_{F'}) = \overline{\tau}(\Gamma_F)$, as subgroups of $\mathcal{G}_7(k)$, and $\zeta_\ell \not\in E' = E F'$. If $F'$ is a good extension then we write $S_{F'}$ for the set of places of $F'$ lying above a place of $S$, $E' = E F'$, and $\widetilde{S}_{F'}$ for the set of places of $E'$ lying above a place $\wv$ (some $v \in S$). If $v \in S_{F'}$ then we write $\wv$ for the unique place of $\widetilde{S}_{F'}$ lying above $v$. We then write $\operatorname{Def}_{F'} : \mathcal{C} \to \operatorname{Sets}$ for the functor which assigns to any $A \in \mathcal{C}$ the set of $\ker(GL(7, A) \to GL(7, k))$-conjugacy classes of homomorphisms $\tau_A : \Gamma_E \to \mathcal{G}_7(A)$ satisfying the following conditions:
	\begin{itemize}
		\item For each place $v | \ell$ of $F'$, $\tau_A|_{\Gamma_{E'_\wv}}$ defines an $A$-point of the `crystalline-ordinary' lifting ring $R_{\overline{\sigma}|_{\Gamma_{E'_\wv}},\{ H_i \}, cr-ord}^{\square}$ described in \cite[\S 1.4]{BGGT}, with sets 
		\[ H_i = \{ 6, 4, 2, 0, -2, -4, -6 \} \]
		 of Hodge--Tate numbers ($i : E'_\wv \to \overline{\QQ}_\ell$ any embedding).
		\item For each place $v | p$ of $F'$, reduction modulo $\ell$ induces an isomorphism $\tau_A(I_{E'_{\wv}}) \to \overline{\tau}(I_{E'_{\wv}})$.
		\item For each finite place $v \nmid \ell p$ of $F'$, $\tau_A|_{\Gamma_{F'_v}}$ is unramified.
		\item $\nu \circ \tau_A = \delta_{E'/F'}$.
	\end{itemize}
	The functor $\operatorname{Def}_{F'}$ is represented by an object $R_{F'} \in \mathcal{C}$ (see \cite[Proposition 2.2.9]{CHT}). We claim that the representation $r_7$ determines a natural map $R_{F'} \to R_{G}$. The only point to check here is that if $\sigma_A : \Gamma_F \to G(A)$ arises from a homomorphism $R_G \to A$, then for any place $v | \ell$ of $F'$, $r_7 \circ \sigma_A|_{\Gamma_{E'_\wv}}$ defines a point of the `cr-ord' lifting ring. This can be reduced to the universal case, meaning we must show that the classifying morphism 
	\begin{equation}\label{eqn_def_rings} R_{r_7 \circ \overline{\sigma}|_{\Gamma_{E'_\wv}}}^\square \to R_{\overline{\sigma}|_{\Gamma_{F'_v}}}^{\square, \check{\delta}^2 \circ \epsilon} 
	\end{equation}
	 associated to $r_7$ (where the first ring is the universal lifting ring for the $GL(7)$-representation $r_7 \circ \overline{\sigma}|_{\Gamma_{E'_\wv}}$, as in \cite[\S 1.2]{BGGT}, and the second ring is the one defined in \cite[Definition 4.1]{Pat16}) factors through the quotient $R_{r_7 \circ \overline{\sigma}|_{\Gamma_{E'_\wv}}}^\square \to R_{\overline{\sigma}|_{\Gamma_{E'_\wv}},\{ H_i \}, cr-ord}^{\square}$. The ring $R_{\overline{\sigma}|_{\Gamma_{F'_v}}}^{\square, \check{\delta}^2 \circ \epsilon}$ is formally smooth over $\cO$, by \cite[Proposition 4.4]{Pat16}, so it suffices to show that the morphism (\ref{eqn_def_rings}) sends the $\overline{\mathbb{Q}}_\ell$-points of $\operatorname{Spec} R_{\overline{\sigma}|_{\Gamma_{F'_v}}}^{\square, \check{\delta}^2 \circ \epsilon} $ into $\operatorname{Spec} R_{\overline{\sigma}|_{\Gamma_{E'_\wv}},\{ H_i \}, cr-ord}^{\square}$. This follows from \cite[Lemma 3.9]{G}.
	 
	  If $F = F'$, the map $R_F \to R_G$ is surjective (same proof as \cite[Lemma 5.7]{BHKT}). For any $F'$, the map $R_{F'} \to R_F \to R_G$ therefore a finite algebra homomorphism \cite[Lemma 1.2.3]{BGGT}. To finish the proof of the theorem, we therefore just need to show there exists a good extension $F' / F$ such that $R_{F'}$ is a finite $W(k)$-algebra.
	
	By \cite[Theorem 10.2]{T}, this will follow if we can find a good extension $F' / F$ and an isomorphism $\iota : \overline{\QQ}_\ell \to \CC$ such that $\overline{\tau}|_{\Gamma_E}$ is the residual representation of an $\iota$-ordinary, cuspidal, regular algebraic, polarizable automorphic representation of $GL(7, \ad_{E'})$. The existence of such a representation follows from \cite[Proposition 3.3.1]{BGGT}. This completes the proof.
\end{proof}
\begin{thm}\label{potaut} Let $\sigma:  \Gamma_F \ra G(\cO)$ be as in Theorem \ref{lifting}. Let
	$$\sigma_7 = r_7 \circ \sigma:  \Gamma_F \ra GL(7,\cO) \hookrightarrow GL(7,\Qlb).$$
	Then there exists a totally real Galois extension $F ' / F$, a cuspidal, cohomological automorphic representation $\Psi(\sigma)$ of $GL(7, \mathbb{A}_{F'})$, and an isomorphism $\overline{\QQ}_\ell \to \CC$ such that $r_\iota(\Psi(\sigma)) \cong \sigma_7|_{\Gamma_{F'}}$. Moreover, we can assume that $\Psi(\sigma)$ is everywhere unramified.
\end{thm}
\begin{proof} It follows from \cite[Corollary 4.5.2]{BGGT} that there exists a totally real Galois extension $F ' / F$, a cuspidal, cohomological automorphic representation $\Psi(\sigma)$ of $GL(7, \mathbb{A}_{F'})$, and an isomorphism $\overline{\QQ}_\ell \to \CC$ such that $r_\iota(\Psi(\sigma)) \cong \sigma_7|_{\Gamma_{F'}}$. Moreover, for each place $v' | \ell$ of $F'$, $\Psi(\sigma)_{v'}$ is unramified. We need to explain why $\Psi(\sigma)$ can be chosen to be everywhere unramified. 

By weak approximation, we can find a totally real, Galois, soluble extension $F'' / F$ such that for each finite place $v'' \nmid \ell$ of $F''$, $\sigma_7|_{\Gamma_{F''_{v''}}}$ is unramified. Let $F^1 = F' \cdot F''$. Then $F^1 / F'$ is a soluble, totally real extension, $F^1 / F$ is Galois, and $\sigma_7|_{\Gamma_{F^1}}$ is unramified at every finite place not dividing $\ell$. We may therefore replace $F'$ by $F^1$ and $\Psi(\sigma)$ by its base change to $F^1$ (which exists, by \cite[Lemma 2.2.2]{BGGT}, and which is everywhere unramified by local-global compatibility).
\end{proof}
We now forget our existing assumptions and restate the above results in a form suitable for application in \S \ref{sec_descent}.
\begin{thm}\label{thm_globalization_of_fixed_parameter}
	Let $K$ be a finite extension of $\QQ_p$, and let $k$ be a finite field of characteristic $\ell \neq p$. Let $\rho : W_K \to G_2(W(k))$ be a continuous homomorphism, irreducible over $\overline{\QQ}_\ell$. Then $\rho(I_K)$ is finite; we assume that $\ell > 28$ and that $\ell$ does not divide the order of $\rho(I_K)$. Then we can find the following data:
	\begin{enumerate}
		\item A totally real number field $L$, together with a non-empty set $\Sigma$ of $p$-adic places of $L$ such that for each $v \in \Sigma$, $L_v \cong K$.
		\item A finite extension $E / \operatorname{Frac} W(k)$ with ring of integers $\cO$ and a continuous representation $\sigma : \Gamma_L \to G_2(\cO)$ of Zariski dense image such that for each place $v \in \Sigma$, $\sigma|_{W_{L_v}}$ and $\rho$ are $G_2(\cO)$-conjugate.
		\item A cuspidal, tempered, regular algebraic, self-dual automorphic representation $\Psi(\sigma)$ of $GL(7, \mathbb{A}_L)$, unramified outside $\Sigma$, and an isomorphism $\iota : \overline{\QQ}_\ell \to \CC$ such that $r_\iota(\Psi(\sigma)) \cong r_7 \circ \sigma$.
	\end{enumerate}
\end{thm}
\begin{proof}
	 We begin with some remarks. First, $\rho$ has finite image, and so extends to a representation of $\Gamma_K$. Indeed, $\rho(I_K)$ is finite since $r_7 \circ \rho$ is semisimple; the $\ell$-adic monodromy theorem implies the existence of a finite extension $K' / K$ such that $r_7 \circ \rho(I_{K'})$ is semisimple and unipotent, hence trivial. Let $\phi \in W_K$ be a Frobenius lift. Since $\rho(I_K)$ is finite, some power $\rho(\phi)^N$ centralizes $\rho(W_K)$. Since $\rho(W_K)$ is irreducible, this forces $\rho(\phi)^N$ to lie in the centre of $G$, hence (since $G$ is adjoint) to be trivial. Thus $\rho(W_K)$ is finite.
	
	Second, let  $\overline{\rho} : W_K \to G_2(k)$ denote the reduction of $\rho$ modulo $\ell$. Then any minimally ramified lift of $\overline{\rho}$ to $G_2(\cO)$ is $G_2(\cO)$-conjugate to $\rho$. Indeed, \cite[Lemma A.2]{FKP} shows that $h^0(\Gamma_K, \mathfrak{g}_k) = 0$, hence (cf. \cite[Lemma 4.17]{Pat16}) that the tangent space to the minimal deformation functor is trivial. This implies that any minimally ramified lift is even $\ker(G_2(\cO) \to G_2(k))$-conjugate to $\rho$. 
	
	By Theorem \ref{potaut}, we can find the following objects:
	\begin{itemize}
		\item A totally real field $F$, and a totally real Galois extension $F' / F$.
		\item A continuous representation $\sigma : \Gamma_F \to G_2(\cO)$ such that $\sigma|_{\Gamma_{F'}}$ has Zariski dense image, $\sigma$ is unramified outside $\ell p$, and for each place $v | \ell$ of $F$, $r_7 \circ \sigma|_{\Gamma_{F_v}}$ is crystalline with Hodge--Tate numbers $\{ 6, 4, 2, 0, -2, -4, -6 \}$ (with respect to any embedding of $F_v$ in $\overline{\QQ}_\ell$).
		\item For each place $v | p$ of $F$, an isomorphism $F_v \cong K$ such that $\sigma|_{\Gamma_{F_v}}$ and $\rho$ are $G_2(\cO)$-conjugate.
		\item A cuspidal, regular algebraic, self-dual automorphic representation $\Psi(\sigma)'$ of $GL(7, \mathbb{A}_{F'})$ which is everywhere unramified and an isomorphism $\iota : \overline{\QQ}_\ell \to \CC$ such that $r_\iota(\Psi(\sigma)') \cong r_7 \circ \sigma|_{\Gamma_{F'}}$.
	\end{itemize}
	Let $v'$ be a place of $F'$ lying above $v$, with $v$ dividing $p$, and let $D_{v' / v} \subset \operatorname{Gal}(F' / F)$ be the decomposition group. We set $F'' = (F')^{D_{v' / v}}$, and $v'' = v'|_{F''}$. Then $F''_{v''} \cong K$ and $\sigma|_{\Gamma_{F''_{v''}}}$ and $\rho$ are $G_2(\cO)$-conjugate. Moreover, $F' / F''$ is a soluble extension. By soluble descent (see e.g.\ \cite[Lemma 2.2.2]{BGGT}), there exists a cuspidal, regular algebraic, self-dual automorphic representation $\Psi(\sigma)''$ of $GL(7, \ad_{F''})$ such that $r_\iota(\Psi(\sigma)'') \cong r_7 \circ \sigma|_{\Gamma_{F''}}$. Finally, we once again use weak approximation, as in the proof of Theorem \ref{potaut}, to find a soluble totally real extension $L / F''$ in which $v''$ splits and such that the base change $\Psi(\sigma)$ of $\Psi(\sigma)''$ to $L$ is unramified away from $v''$. The proof is complete on taking $\Sigma$ to be the set of places of $L$ lying above $v''$, and noting that \cite[Corollary 1.3]{S} shows that $\Psi(\sigma)$ is tempered.
\end{proof}

\section{Automorphic descent from $GL(7)$ to $G_2$}\label{sec_descent}

The following theorem was recently proved by Hundley and Liu:

\begin{thm}\label{HLiu} Let $F$ be a number field, and let $\Psi$ be a cuspidal automorphic representation of $GL(7,\ad_{F})$.  Suppose that the following conditions hold.
\begin{enumerate}
\item For almost all places $v$ of $F$ at which $\Psi_v$ is unramified, the Satake parameter of the local component $\Psi_v$ is conjugate, in $GL(7, \CC)$, to an element of $r_7(G_2(\CC))$.
\item The partial $L$-function $L^S(s,\Psi,\wedge^3)$ has a pole at $s = 1$, for some finite set $S$.  
\end{enumerate}
Then there exists a globally generic  cuspidal automorphic representation $\Pi$ of $G_2(\ad_F)$ such that for all but finitely many places at which $\Pi_v$ is unramified, the Satake parameter of $\Psi_v$ is the image under $r_7$ of the Satake parameter of $\Pi_v$. 

Moreover, if $v$ is a finite place at which $\Psi_v$ is both unramified and tempered, with Satake parameter conjugate to an element of $r_7(G_2(\CC))$, then $\Pi_v$ is unramified and the Satake parameter of $\Psi_v$ is the image under $r_7$ of the Satake parameter of $\Pi_v$.
\end{thm}
\begin{proof}
	The theorem up to ``moreover'' is contained in \cite[Theorem 6.1.17]{HL} \cite[Theorem 6.4.31]{HL} and \cite[Theorem 6.5.6]{HL} (giving genericity, cuspidality and weak lifting, respectively). If $v$ is any finite place at which $\Psi_v$ is unramified with Satake parameter conjugate to $r_7(G_2(\CC))$, then \cite[Proposition 6.5.5]{HL} shows that $\Pi_v$ is the generic subquotient of the unramified principal series representation of $G_2$ with Satake parameter corresponding to that of $\Psi_v$. If $\Psi_v$ is tempered then this unramified principal series representation is also tempered and irreducible, implying the final claim.
\end{proof}
We remark that if $F$ is totally real and $\Psi$ is cuspidal and regular algebraic with infinitesimal character ``integral of $G_2$ type'', then $\Pi_\infty$ must be discrete series. Indeed, the theta lift of $\Pi$ is an automorphic representation of $PSp(6)$ whose restriction to $Sp(6)$ gives rise by functorial transfer to an automorphic representation $\Psi'$ of $GL(7, \ad_F)$ with the property  that the infinitesimal character of $\Psi'_\infty$ is obtained from that of $\Pi_\infty$ by functoriality
(cf. the discussion in \S \ref{thetaR}). By strong multiplicity one, $\Psi = \Psi'$, and the infinitesimal character of $\Pi_\infty$ is integral. Since $\Pi$ is globally generic, fact \ref{inf} implies that $\Pi_\infty$ is discrete series. 
\begin{corollary}\label{descent} Let $L$ \textup{(}a totally real number field\textup{)}, $\sigma: \Gamma_L \to G_2(\overline{\QQ}_\ell)$ \textup{(}a continuous homomorphism\textup{)} and $\Psi(\sigma)$ \textup{(}a cuspidal tempered automorphic representation of $GL(7, \ad_L)$ with $r_7 \circ \sigma \cong r_\iota(\Psi(\sigma))$\textup{)} be as in the statement of Theorem \ref{thm_globalization_of_fixed_parameter}. Then there exists a globally generic, cuspidal automorphic representation $\Pi(\sigma)$ of $G_2(\ad_L)$ such that $\Pi(\sigma)_\infty$ is discrete series and for every place $v$ of $L$ at which $\Psi(\sigma)$ is unramified, $\Pi(\sigma)_v$ is unramified and the Satake parameter of $\Psi(\sigma)_v$ is the image in $GL(7, \CC)$ of the Satake parameter of $\Pi(\sigma)_v$.
\end{corollary}
\begin{proof}  Since the local parameter of $\Psi(\sigma)$ at every  place at which $\sigma$ is unramified factors through $r_7(G_2)$, $\Psi(\sigma)$ satisfies  hypothesis (i) of Theorem \ref{HLiu}. If hypothesis (ii) is satisfied, then $\Psi(\sigma)$ is the functorial lift of some cuspidal $\Pi(\sigma)$. Finally, the temperedness of $\Psi(\sigma)$ will imply that $\Psi(\sigma)_v$ is the lift of $\Pi(\sigma)_v$ for all unramified places $v$.
	
It remains to verify that (ii) is satisfied. 
Now $L^S(s,\Psi(\sigma),\wedge^3) = L^S(s,\wedge^3\circ\sigma_7)$ if $S$ contains all ramified places.  
Moreover, since $\sigma_7 = r_7 \circ \sigma$ factors through $r_7(G_2)$, it is well known that
\begin{equation}\label{wedge}  \wedge^3\circ \sigma_7 \isoarrow \sigma_7 \oplus Sym^2\circ \sigma_7.
\end{equation}
Thus 
$$L^S(s,\Psi(\sigma),\wedge^3) = L^S(s,\Psi(\sigma))\cdot L^S(s,\Psi(\sigma),Sym^2).$$
The first factor does not vanish at $s = 1$ by the theorem of Jacquet--Shalika and Shahidi, whereas the second factor has a simple pole at $s = 1$ because $\Psi(\sigma)$ is a self-contragredient representation of an odd general linear group (cf. the discussion of \cite[p. 139]{BG}).
\end{proof}

\begin{thm}\label{surj}  Let $K$ be a $p$-adic local field. Then the map
\[	 \CL_g :  \CA^0_g(G_2,K) \rightarrow \CG^0(G_2,K) \]
	 is surjective.
\end{thm}
\begin{proof} 
	Let $\rho : W_K \to G_2(\CC)$ be an irreducible representation. We will construct $\pi \in \CA^0_g(G_2,K)$ such that $\CL_g(\pi) = \rho$.	Fix a prime $\ell > 28$ which does not divide the (finite) order of $\rho(W_K)$, and let $\iota : \overline{\QQ}_\ell \to \CC$ be an isomorphism. We may assume that $\iota^{-1} \rho$ takes values in $G_2(\ZZ_\ell)$. We take the totally real number field $L$ and homomorphism $\sigma : \Gamma_L \to G_2(\cO)$ as in Theorem \ref{thm_globalization_of_fixed_parameter}, and apply Corollary \ref{descent}. Thus there is a non-empty set $\Sigma$ of $p$-adic places of $L$, and for each $v \in \Sigma$ an isomorphism $K \cong L_v$ such that $\iota^{-1}\rho, \sigma|_{\Gamma_{L_v}}$ are $G_2(\cO)$-conjugate. Fix a choice of $v \in \Sigma$, and let $\pi = \Pi(\sigma)_v$. We must show that  $\pi$ is supercuspidal and that $\CL_g(\pi) = \rho$.
	
	We first note that $r_\iota(\Pi(\sigma))$ (defined as in the proof of Proposition \ref{prop_image_in_G_2}) is $G_2(\overline{\QQ}_\ell)$-conjugate to $\sigma$; indeed, this can be checked at unramified places, so follows from \cite{Gri95}. We therefore just need to check that $\pi$ is supercuspidal. If $\theta_7(\pi)$ is supercuspidal, then Proposition 3.4 of \cite{SWe} shows that $\pi$ is also supercuspidal.

	We can therefore assume that $\theta_7(\pi)$ is not supercuspidal.  	 Thus its Jacquet module $J_P(\theta_7(\pi)) \neq 0$ for some maximal parabolic subgroup $P \subset PSp(6)$.   It follows from \cite[Proposition 7.5]{CKPS} that the image of the local parameter $r_7 \circ \rho$ is contained in a proper parabolic subgroup of $SO(7)$.  	However, since the image of $\rho$ is not contained in any proper parabolic of $G_2$, it  follows from
\cite[Proposition 1.1]{SWe} that  that $P$ must be the Siegel parabolic subgroup, with Levi quotient $GL(3)$.   In particular, $J_P(\theta_7(\pi))$ must be
supercuspidal.

We now turn to the proof of \cite[Proposition 3.6]{SWe}, specifically the discussion labelled (Siegel) on p. 759, which computes $J_P(\theta_7(\pi))$.  Although
the running hypothesis in that Proposition is that $\pi$ is supercuspidal, this hypothesis is only used in the first paragraph, which cites \cite[Theorem 5.3]{MS} to prove
that, (modifying the notation of \cite{MS}), $J_P(\theta_7(\pi))\boxtimes \pi$ is a subrepresentation of $(\Pi_7)_{U_7}$.  Here $\Pi_7$ is the minimal representation
of the adjoint split group $\mathbf{E_7}$ over the field $K$, and $U_7$ is the unipotent radical of a specific maximal parabolic subgroup of $\mathbf{E_7}$ denoted $\mathbf{Q_7}$ (with Levi quotient isomorphic to $\mathbf{E_6}$).  

Suppose $J_P(\theta_7(\pi))\boxtimes \pi$ is {\it not} a subrepresentation of $(\Pi_7)_{U_7}$.  Then the discussion on p. 759 of \cite{SWe} implies that 
$J_P(\theta_7(\pi))\boxtimes \pi$ intertwines with one of the terms (1) or (2) of \cite[Theorem 5.3]{MS}.  But  we have seen that $J_P(\theta_7(\pi))$
is supercuspidal, and this is incompatible with the description of these terms, which are  induced from proper parabolic subgroups of $GL(3)$.  

Thus $J_P(\theta_7(\pi)) \boxtimes \pi$ is a non-zero subrepresentation of $(\Pi_7)_{U_7}$.  The second displayed formula on p. 759 of \cite{SWe} then asserts that $\theta_6(\pi) \neq 0$, and moreover that $\theta_6(\pi)$ is a supercuspidal representation of $PGL(3,K)$.     But then \cite[Theorem 21]{GS04}
implies that $\pi$ is supercuspidal. 

\end{proof}

\section{Injectivity}

Let $K$ be a $p$-adic local field. The final step in our argument is to prove that the parameterization
\[ \CL_g :  \CA^0_g(G_2,K) \rightarrow \CG^0(G_2,K) \]
is injective (hence bijective). We recall that if $\pi \in \CA^0_g(G_2,K)$, then $\theta_7(\pi)$ is a representation of $PSp(6, K)$. Its restriction to $Sp(6, K)$ contains a unique generic constituent $\sigma$, which lifts to $GL(7, K)$ by Theorem \ref{thm_Arthur}, and $\CL_g(\pi)$ is defined to be  the unique $G_2$-valued parameter which is conjugate in $GL(7, \CC)$ to the Galois parameter of the lift $\Psi(\sigma)$ to $GL(7, K)$. 

The lifting from generic representations of $Sp(6, K)$ to representations of $GL(7, K)$ is injective. On the other hand, Savin--Weissman study the injectivity of the lift to $PSp(6, K)$ (see \cite[Theorem 4.7]{SWe}). The problematic step for us is therefore restriction from $PSp(6, K)$ to $Sp(6, K)$. Fortunately this has been analyzed by Xu \cite{Xu}, and we will be able to prove injectivity by combining these results with a global argument. 
\begin{lemma}\label{lem_twisting}
    Let $\pi, \pi' \in \CA^0_g(G_2,K)$, and suppose that $\CL_g(\pi) = \CL_g(\pi')$. Then there exists a quadratic character $\omega : PSp(6, K) \to \CC^\times$ such that $\theta_7(\pi) \cong \theta_7(\pi') \otimes \omega$. 
\end{lemma}
\begin{proof}
	Let $\sigma$ be the unique generic constituent of the restriction of $\theta_7(\pi)$ to $Sp(6, K)$, and define $\sigma'$ similarly. Then our hypothesis implies $\Psi(\sigma) = \Psi(\sigma')$, hence (Theorem \ref{thm_Arthur}) $\sigma = \sigma'$. Then \cite[Corollary 6.4]{Xu} shows the existence of a character $\omega$ with the claimed property. 
\end{proof}
\begin{corollary}\label{cor_fibres}
	Let $\rho \in \CG^0(G_2,K)$. Then the following are equivalent:
	\begin{enumerate}
		\item $\CL_g^{-1}(\rho)$ is a singleton.
		\item There exists $\pi \in \CA^0_g(G_2,K)$ such that $\CL_g(\pi) = \rho$,
		with the following property: for any non-trivial character $\omega : PSp(6, K) \to \CC^\times$, either (a) $\theta_7(\pi) \cong \theta_7(\pi) \otimes \omega$ or (b) $\theta_7(\pi) \otimes \omega$ is not of the form $\theta_7(\pi')$ for any $\pi' \in \CA^0_g(G_2,K)$.
	\end{enumerate}
\end{corollary}
\begin{proof}
	Suppose that (i) holds, and choose any $\pi$ with $\CL_g(\pi) = \rho$. If $\pi' \in \CA^0_g(G_2,K)$ and $\theta_7(\pi) \otimes \omega \cong \theta_7(\pi')$ then (as $\CL_g$ is injective) we have $\pi = \pi'$, hence $\theta_7(\pi) \cong \theta_7(\pi) \otimes \omega$.
	
	Suppose instead that (ii) holds, and let $\pi', \pi'' \in \CA^0_g(G_2,K)$ be representations with $\CL_g(\pi') = \CL_g(\pi'') = \rho$. Let $\pi \in \CA^0_g(G_2,K)$ be the given representation with $\CL_g(\pi) = \rho$. Thus there exist characters $\omega', \omega''$ such that $\theta_7(\pi') \cong \theta_7(\pi) \otimes \omega'$, $\theta_7(\pi'') \cong \theta_7(\pi) \otimes \omega''$. Consideration of possibilities (a) and (b) shows that $\theta_7(\pi) = \theta_7(\pi')$. By symmetry, we also have $\theta_7(\pi) = \theta_7(\pi'')$, hence $\theta_7(\pi') = \theta_7(\pi'')$, hence $\pi' = \pi''$ by \cite[Theorem 4.7]{SWe}. 
\end{proof}
We first dispense with the simplest case.
\begin{proposition}\label{prop_not_supercuspidal}
	Let $\pi, \pi' \in \CA^0_g(G_2,K)$, and suppose that $\CL_g(\pi) = \CL_g(\pi')$. Suppose further that $\theta_7(\pi)$ is not supercuspidal. Then $\pi = \pi'$. 
\end{proposition}
\begin{proof}
	Our hypotheses imply that $\theta_7(\pi')$ is also not supercuspidal. By \cite[Theorem 3.9]{SWe} there exist supercuspidal representations $\tau, \tau'$ of $GL(3, K)$ of trivial central character such that $\theta_7(\pi)$ is the generic subquotient of $i_{Q_3}^{PSp(6)} \tau$, and similarly for $\theta_7(\pi')$ (and we use normalized induction). The existence of an isomorphism $\theta_7(\pi) \otimes \omega \cong \theta_7(\pi')$ implies that $\tau \otimes \omega$, $\tau'$ are conjugate under the stabilizer in the Weyl group of $PSp(6)$ of the Levi subgroup of $Q_3$ (uniqueness of supercuspidal support). This in turn implies that $\tau \otimes \omega$ is isomorphic to one of $\tau'$ or $(\tau')^\vee$. This is a contradiction, because the central character of $\tau \otimes \omega$ is non-trivial. 
\end{proof}
\begin{lemma}\label{lem_shahidi_criterion}
	To show that $\CL_g$ is injective, it is enough to prove the following statement:
	\begin{itemize}
		\item Let $\rho \in \CG^0(G_2,K)$ be a parameter which remains irreducible in $SO(7)$. Then there exists a representation $\pi \in \CA^0_g(G_2,K)$ such that $\CL_g(\pi) = \rho$ and for any character $\omega : PSp(6, K) \to \CC^\times$ such that $\theta_7(\pi) \not\cong \theta_7(\pi) \otimes \omega$, the Shahidi $L$-function $L(\theta_7(\pi) \otimes \omega, Spin, s)$ is holomorphic at $s = 0$.
	\end{itemize}
\end{lemma}
The Shahidi $L$-function is the one defined in \cite[\S 5.1]{SWe} using the realisation of $GSp(6)$ as a Levi subgroup of $F_4$.
\begin{proof}
	By Proposition \ref{prop_not_supercuspidal}, it's enough to show that the condition (ii) of Corollary \ref{cor_fibres} holds for parameters $\rho$ which are irreducible in $SO(7)$. Then \cite[Theorem 5.10]{SWe} and \cite[Proposition 4.6]{SWe} together show that the holomorphy of $L(\theta_7(\pi) \otimes \omega, Spin, s)$ at $s = 0$ implies that $\theta_7(\pi) \otimes \omega$ does not have the form $\theta_7(\pi')$ for any $\pi' \in \CA^0_g(G_2,K)$.
\end{proof}
\begin{proposition}\label{prop_global_data_implies_injectivity}
	To show that $\CL_g$ is injective, it is enough to construct for each $\rho \in \CG^0(G_2,K)$ which remains irreducible in $SO(7)$ the following data:
	\begin{enumerate}
		\item A totally real field $F$, together with a non-empty set $\Sigma$ of $p$-adic places of $F$ and for each $v \in \Sigma$ an isomorphism $F_v \cong K$.
		\item A cuspidal, globally generic automorphic representation $\pi$ of $G_2(\ad_F)$ which is unramified outside $\Sigma$ and discrete series at infinity.
		\item For each character $\omega : PSp(6, K) \to \CC^\times$, a character 
		$$\Omega : PSp(6, F) \backslash PSp(6, \ad_F) \to \CC^\times$$ which is unramified away from $\Sigma$ and satisfies $\Omega|_{PSp(6, F_v)} = \omega$ for each $v \in \Sigma$.
	\end{enumerate}
	with the following property:
	\begin{enumerate}
		\item[(iv)] Let $\Psi$ be the lift of $\pi$ to $GL(7, \ad_F)$. Then $\Psi$ is cuspidal and there exists a prime $\ell \neq p$ and an isomorphism $\iota : \overline{\QQ}_\ell \to \CC$ such that $r_\iota(\Psi)|_{W_{F_v}} \cong \iota^{-1} r_7 \circ \rho$ for all $v \in \Sigma$.
	\end{enumerate}
\end{proposition}
\begin{proof}
	Recall that $\theta_7(\pi)$ is the globally generic, cuspidal lift of $\pi$ to $PSp(6, \ad_F)$, and $\Psi$ is the lift of the globally generic constituent of $\theta_7(\pi)|_{Sp(6, \ad_F)}$ to $GL(7, \ad_F)$. Fix $v_0 \in \Sigma$, and let $\omega : PSp(6, F_{v_0}) \to \CC^\times$ be a character such that $\theta_7(\pi_{v_0}) \not\cong \theta_7(\pi_{v_0}) \otimes \omega$. We must show that the Shahidi L-function $L(\theta_7(\pi_{v_0}) \otimes \omega, Spin, s)$ is holomorphic at $s = 0$ (as then the criterion of Lemma \ref{lem_shahidi_criterion} will be satisfied with $\CL_g(\pi_{v_0}) = \rho$). We note that $\theta_7(\pi_{v})$ is supercuspidal for each $v \in \Sigma$, by Proposition \ref{prop_supercusp}.
	
Invoking \cite[Theorem 3.5]{Sh}, we have an identity
\begin{equation}\label{eqn_gamma_factors} \gamma(\theta_7(\pi)_\infty \otimes \Omega_\infty, s) \prod_{v \in \Sigma} \gamma(\theta_7(\pi_v) \otimes \omega, s) = \frac{L^\Sigma(\theta_7(\pi) \otimes \Omega, Spin, s)}{L^\Sigma(\theta_7(\pi) \otimes \Omega, Spin, 1-s)}, \end{equation}
where the $\gamma$-factors are those of \emph{loc. cit.} and $L^\Sigma$ denotes the prime-to-$\Sigma$-and-$\infty$ $L$-function, which is therefore a product of unramified local $L$-factors. 

By \cite[Proposition 7.3]{Sh}, the factors $\gamma(\theta_7(\pi_v) \otimes \omega, s)$ ($v \in \Sigma$) are rational functions of $q_v^{-s}$ which are holomorphic at $s = 0$, and which vanish at $s = 0$ if and only if $L(\theta_7(\pi_v) \otimes \omega, Spin, s)$ has a pole at $s = 0$. To prove the proposition, it will therefore be enough for us to show that $\prod_{v \in \Sigma} \gamma(\theta_7(\pi_v) \otimes \omega, s)$ does not vanish at the point $s = 0$. The zeroes and poles of $\gamma(\theta_7(\pi)_\infty \otimes \Omega_\infty, s)$ lie on finitely many lines parallel to the real axis, so it is even enough for us to show that the expression in (\ref{eqn_gamma_factors}) does not vanish at infinitely many points of the form $s = 2 \pi i k / \log q_{v_0}$ ($k \in \mathbb{Z}$). This is what we will now do.

For any place $v$ of $F$, we may identify the quotient $PSp(6, F_v) / Sp(6, F_v)$ with $F_v^\times / (F_v^\times)$; we may further identify $\Omega$ with a quadratic character of $\GL_7(\ad_F)$, by composition with the determinant. Using the computation of unramified $L$-functions, the right-hand side of (\ref{eqn_gamma_factors}) equals the quotient
	\[ \frac{L^\Sigma(\Psi \otimes \Omega, s) L^\Sigma(\Omega, s)}{L^\Sigma(\Psi \otimes \Omega, 1-s) L^\Sigma(\Omega, 1-s)} \]
	of standard $L$-functions. Using the functional equation for these standard $L$-functions, we find that this equals
	\[ \prod_{v \in \Sigma} \frac{L_v(\Psi \otimes \Omega, 1-s) L_v(\Omega, 1- s)}{L_v(\Psi \otimes \Omega, s) L_v(\Omega, s)}, \]
	up to product with a meromorphic function with all of its zeroes and poles on the real axis. Since $\omega$ is non-trivial, each factor $L_v(\Omega, 1-s) / L_v(\Omega, s)$ is holomorphic on the line $Re(s) = 0$, with zeroes only possible at the points of $(\log q_v)^{-1} ( \pi i + 2 \pi i \mathbb{Z})$. If $v \in \Sigma$ then the representation $\CL_g(\pi_v)$ has finite image, so $L_v(\Psi \otimes \Omega, 1-s) = L_v( (r_7 \circ \rho) \otimes \omega, 1-s)$ is  holomorphic and non-vanishing on the line $Re(s) = 0$. We are therefore reduced to showing that 
	\[ \prod_{v \in \Sigma} L_v(\Psi \otimes \Omega, s) = L_{v_0}((r_7 \circ \rho) \otimes \omega, s)^{| \Sigma |} \]
	does not have poles at infinitely many points of the form $s = 2 \pi i k / \log q_v$ ($k \in \ZZ$). Equivalently, $(r_7 \circ \rho) \otimes \omega$ does not contain the trivial representation of $W_K$.
	
	Now we make use of \cite[Corollary 4.2]{Xu}. Writing $R : G_2 \to Spin(7)$ for the natural homomorphism, it states that $\theta_7(\pi_{v_0}) \cong \theta_7(\pi_{v_0}) \otimes \omega$ if and only if $R \circ \rho$, $(R \circ \rho) \otimes \omega$ are $Spin(7)$-conjugate. We are assuming that this is not the case. Lemma \ref{lem_G_2_conjugacy} below implies that $(r_7 \circ \rho) \otimes \omega$ does not contain the trivial representation; and this completes the proof of the proposition.
\end{proof}
\begin{lemma}\label{lem_G_2_conjugacy}
	Let $\Gamma$ be a group, and let $\rho : \Gamma \to G_2(\CC)$ be a completely reducible representation and $\omega : \Gamma \to \{ \pm 1 \}$ a non-trivial character such that $R \circ \rho$ and $(R \circ \rho) \otimes \omega$ are not $Spin(7)$-conjugate. Then $(r_7 \circ \rho) \otimes \omega$ does not contain the trivial representation.
\end{lemma}
\begin{proof}
	Let $R' : Spin(7) \to GL(8)$ be the spin representation, $R'' : Spin(7) \to GL(7)$ the vector representation. Then $R' \circ R = r_7 \oplus \CC$ and $R'' \circ R = r_7$. In particular, $R' \circ ((R \circ \rho )\otimes \omega) = (r_7 \circ \rho) \otimes \omega \oplus \omega$. We see that $(r_7 \circ \rho) \otimes \omega$ contains the trivial representation if and only if $(R \circ \rho) \otimes \omega$ fixes a non-zero vector in the spin representation.
	
	Since the stabilizers of non-zero vectors in the spin representation are exactly the $Spin(7)$-conjugates of $R(G_2)$, we see that $(r_7 \circ \rho) \otimes \omega$ contains the trivial representation if and only if there exists $g \in Spin(7, \CC)$ such that $\operatorname{Ad}(g) \circ ((R \circ \rho) \otimes \omega)$ is valued in $R(G_2(\CC))$. 
	
	Suppose for contradiction that this is case, and let $\rho' : \Gamma \to G_2(\CC)$ be the homomorphism such that $R \circ \rho' = \operatorname{Ad}(g) \circ ((R \circ \rho) \otimes \omega)$. Then $r_7 \circ \rho' = \operatorname{Ad}(R''(g)) \circ r_7 \circ \rho$. By \cite{Gri95}, this implies that $\rho, \rho'$ are themselves $G_2$-conjugate, hence that $R \circ \rho$ and $R \circ \rho'$ are $R(G_2)$-conjugate, hence that $R \circ \rho$ and $(R \circ \rho) \otimes \omega$ are $Spin(7)$-conjugate. This contradicts our hypothesis.
\end{proof}
To complete the proof of injectivity, it remains to construct data as in Proposition \ref{prop_global_data_implies_injectivity}. Given an irreducible representation $\rho : W_K \to G_2(\CC)$ which remains irreducible in $SO(7)$, Theorem \ref{thm_globalization_of_fixed_parameter} implies the existence of the following data:
\begin{enumerate}
\item A totally real field $F$, together with a non-empty set $\Sigma$ of $p$-adic places of $F$ and for each $v \in \Sigma$ an isomorphism $F_v \cong K$.
\item A prime $\ell \neq p$, an isomorphism $\iota : \overline{\QQ}_\ell \to \CC$, and a homomorphism $\sigma : \Gamma_F \to G_2(\overline{\QQ}_\ell)$ of Zariski dense image such that for each $v \in \Sigma$, $\sigma|_{W_{F_v}}$ is conjugate to $\iota^{-1} \rho$.
	\item A cuspidal, regular algebraic, self-dual automorphic representation $\Psi$ of $GL(7, \ad_F)$ such that $r_\iota(\Psi) \cong r_7 \circ \sigma$ and $\Psi$ is unramified outside $\Sigma$.
\end{enumerate}
We choose for each character $\omega : PSp(6, K) \to \CC^\times$ a globalization $\Omega$ such that for each $v \in \Sigma$, the restriction of $\Omega$ to $PSp(6, F_v)$ equals $\omega$. After making a quadratic base change, split at $\Sigma$, we can assume moreover that each character $\Omega$ is unramified away from $\Sigma$.

Applying Theorem \ref{HLiu}, we obtain a cuspidal, globally generic representation $\pi$ of $G_2(\ad_F)$, unramified outside $\Sigma$ and discrete series at infinity. We have now constructed all of the required data.
\section{Final remarks}
 It is possible to use the same strategy of passing to $GL(7)$, combined with properties of $L$-functions, to show that no pure generic supercuspidal representation of $G_2$ is {\it incorrigible}, in the sense of \cite{H18}.  But this can also be derived from the dichotomy of Savin and Weissman \cite{SWe} and any proof using $L$-functions ultimately reduces to the dichotomy property.

\appendix
\section{Genericity of a lift by Gordan Savin}

 Let $F$ be a global field and $\mathbb A$  its ring of ad\`eles. The goal of the appendix is to show that any generic cuspidal automorphic 
form on $G_2(\mathbb A)$ lifts to a generic automorphic form on $\PGSp_6(\mathbb A)$. 

\subsection{Octonions and $G_2$}

We follow the exposition \cite{SW15} and work over $\mathbb Q$. Let $\mathbb H$ be the algebra of Hamilton quaternions, with the usual basis $\{1,i,j,k\}$. 
The split octonion algebra 
over $\mathbb Q$ is 
$\Oct =\mathbb H \oplus \mathbb H$ with multiplication  
\[ 
(a,b) \cdot (c,d) = (ac+d\bar b, \bar a d+cb). 
\] 
If $x=(a,b)$, let $\bar x= (\bar a, -b)$. 
Then  $x\mapsto \bar{x}$ is  a linear anti-involution of $\Oct$, defining norm and trace maps 
 $$ \N:\Oct\to F\quad x\mapsto x\bar{x} = \bar{x}x,\qquad \Tr:\Oct\to F,\quad x\mapsto x+\bar{x} $$
satisfying
 $$ \N(x\cdot y) = \N(x)\N(y),\qquad \Tr(x\cdot y) = \Tr(y\cdot x), \quad \Tr(x\cdot(y\cdot z)) = \Tr((x\cdot y)\cdot z). $$ 
 Let $l=(0,1)$, so $\Oct$ has a basis $\{1,i,j,k,l,li,lj,lk\}$.
The following basis is particularly useful.
\begin{equation}\label{eq:Octbasis}
\begin{array}{c} s_1 = \frac{1}{2}(i+li),\quad s_2 = \frac{1}{2}(j+lj),\quad s_3 = \frac{1}{2}(k+lk),\quad s_4 = \frac{1}{2}(1+l), \\
 t_1 = \frac{1}{2}(i-li),\quad t_2 = \frac{1}{2}(j-lj),\quad t_3 = \frac{1}{2}(k-lk),\quad t_4 = \frac{1}{2}(1-l). \end{array}
\end{equation}
The multiplication table for this basis is given in Table~\ref{table}.
\begin{table}[h] 
$$\begin{array}{||c||c|c|c||c|c|c||c|c|} \hline
 & s_1 & s_2 & s_3 & t_1 & t_2 & t_3 & s_4 & t_4 \\ \hline \hline
s_1 & 0 & -t_3 & t_2 & s_4 & 0 & 0 & 0 & s_1 \\ \hline
s_2 & t_3 & 0 & -t_1 & 0 & s_4 & 0 & 0 & s_2 \\ \hline
s_3 & -t_2 & t_1 & 0 & 0 & 0 & s_4 & 0 & s_3 \\ \hline \hline
t_1 & t_4 & 0 & 0 & 0 & s_3 & -s_2 & t_1 & 0 \\ \hline
t_2 & 0 & t_4 & 0 & -s_3 & 0 & s_1 & t_2 & 0 \\ \hline
t_3 & 0 & 0 & t_4 & s_2 & -s_1 & 0 & t_3 & 0 \\ \hline \hline
s_4 & s_1 & s_2 & s_3 & 0 & 0 & 0 & s_4 & 0 \\ \hline
t_4 & 0 & 0 & 0 & t_1 & t_2 & t_3 & 0 & t_4 \\ \hline
\end{array}$$
\caption{Multiplication Table for Octonions}
\label{table}
\end{table}

\smallskip 

Let $\mathcal R \subset \Oct$ be the $\mathbb Z$-lattice spanned by $s_i$ and $t_i$. It follows, from the multiplication table, that 
$\mathcal R$ is an order.  It is maximal since the determinant of the trace pairing $\Tr(x\cdot y)$
on $\mathcal R$ is 1. Let $\Oct^0$ be the subspace of trace zero elements.  For every subspace $V \subset \Oct^0$, 
let  $V^{\Delta}$ be the subspace of  all $x\in \Oct^0$ such that $x\cdot y=0$ for all $y\in V$. 
A subspace $V\subset \Oct^0$ on which multiplication is trivial is at most 2-dimensional.  
(We call such a subspace a \emph{null space} or a \emph{null subspace}.)  Indeed, let $(i,j,k)$ be a permutation of $(1,2,3)$.  
Then from the multiplication table we see that $\langle s_i\rangle^{\Delta} = \langle s_i,t_j,t_k\rangle$, and the null spaces of $\Oct^0$ which contain $s_i$ are all of the form $\langle s_i,at_j+bt_k\rangle$ for fixed $a,b\in \mathbb Q$.  
Since $G_2$, the group of automorphisms of $\Oct$, acts transitively on (nonzero) elements of trace zero and norm zero, this phenomenon is generic.

\smallskip 
The group $G_2$ has two conjugacy classes of maximal parabolic subgroups, 
and they can be described as the stabilizers of null subspaces in $\Oct^0$.  Let $V_1 \subset V_2$ be $1$ and $2$-dimensional 
null subspaces. Let $V_3= V_1^{\Delta} \supset V_2$.  Let 
$Q_1=L_1 U_1$ and $Q_2=L_2 U_2$ be the stabilizers of $V_1$ and $V_2$, respectively. The Levi factors $L_1$ and $L_2$ are 
isomorphic to $\GL(V_3/V_1)$ and $\GL(V_2)$, respectively. The Borel subgroup  $Q_0=L_0U_0 = Q_1 \cap Q_2$ stabilizes the 
full flag 
\[ 
V_1 \subset V_2 \subset V_3 \subset V_3^{\perp}  \subset V_2^{\perp}  \subset V_1^{\perp}  
\] 
in $\Oct^0$, where $V^{\perp}$ stands for the orthogonal complement of $V$ with respect to the trace pairing.

\subsection{Albert algebra and $E_7$}  This is an exceptional $27$-dimensional Jordan algebra $J$ over $\mathbb Q$.  It is the set of matrices 
\[ 
A=\left(\begin{array}{ccc} 
\gamma & x & \bar y \\
\bar x & \beta & z \\
y & \bar z & \alpha
\end{array}\right)
\] 
where $\alpha,\beta,\gamma\in \mathbb Q$ and $x,y,z\in \mathbb O$. We have a cubic form (the determinant) on $J$
\[ 
 \det A= \alpha\beta\gamma  - \alpha \N(x) - \beta \N(y) - \gamma\N(z) + \Tr(xyz). 
  \] 
 The group of isogenies of the cubic form is a reductive group of type $E_6$.  Its orbits on 
  $J$ are classified by the rank of the matrix $A$.  If $A\neq 0$ then $A$ has rank one if $A^2= \Tr(A) \cdot A$. Explicitly, this means that the entries of $A$ satisfy the equalities 
\[ 
\N(x)= \beta\gamma, \, \N(y)=\gamma\alpha, \, \N(z)=\alpha\beta, \, \alpha\bar x=zy, \, \beta \bar y = x z ,\,  \gamma  \bar z= yx. 
\] 

\smallskip Let $G$ be the split, adjoint group of type $E_7$. This group can be constructed from $J$ by the Koecher-Tits construction, see Section 3 in \cite{KS15}. 
 In particular, $G$ has a  pair of opposite maximal parabolic subgroups $P=MN$ and $\bar P=M\bar N$, $N\cong J$  and $\bar N \cong J$, 
 such that the conjugation action of $M$ on $N$ 
 (this action is faithful since $G$ is adjoint) gives an isomorphism 
of $M$ and the group of similitudes of the cubic form on $J$: 
\[ 
M\cong \{ g\in \GL(J) ~|~ \text{ for some $\lambda\in F^{\times}$,} \det( gA) =\lambda \det(A) \text{ for all $A\in J$}\}.
\]
The action of $M$ on $J$ resulting from the isomorphism 
 $\bar N \cong J$ is dual to the action arising from $N$. 
Observe that $G_2$ acts naturally on $J$ (by acting on off-diagonal entries). This gives an embedding of $G_2$ into $M$.  Let $M_3$,  $N_3$  and $\bar N_3$ 
be the centralizers of $G_2$ in $M$ and $N$, respectively. It is clear that $N_3\cong J_3$ and $\bar N_3\cong J_3$, where $J_3$ is the Jordan algebra of symmetric $3\times 3$-matrices. 
The group $M_3$ is isomorphic to $\GL_3$. This isomorphism is realized by observing that $\GL_3$ acts on $J$ by similitudes 
\[
A\mapsto \det(g)^{-1} g A g^t 
\] 
for all $g\in \GL_3$. Thus the centralizer of $G_2$ in $G$ is $\PGSp_6$ with $P_3=M_3 N_3$ and $\bar P_3=M_3\bar N_3$ a pair of opposite  maximal parabolic subgroups. 

\subsection{Minimal representation of $G$}
 Let $F$ be a local field of characteristic 0. In this section $J$, $G$ etc stand for their sets of $F$-points. 
  Fix $\psi : F \rightarrow \mathbb C^{\times}$, a non-trivial unitary additive character. 
 Every $A\in J$ defines a character  $\psi_A$ of $J$ 
 \[ 
\psi_A(B) = \psi (\Tr(A\circ B)) = \psi_B(A) 
\] 
for all $B\in J$, where $A\circ B$ denotes the Jordan multiplication. Via the isomorphism $N\cong J$ we view $\psi_A$ a character of $N$, 
and every unitary character of $N$ is equal to $\psi_A$ for some $A$.  
Let $\Omega \subseteq J$ be the set of rank one elements in $J$. We view $\Omega \subset \bar N$, where $\bar N$ is the opposite of $N$.  
A unitary model of the minimal representation is $L^2(\Omega)$. Here only the acton of the maximal parabolic $P=MN$ is obvious: Let $n\in N$ correspond to 
$B\in J$ via the isomorphism $N\cong J$. Then, for $f\in  L^2(\Omega)$, 
\[
\pi(n)f(A) = \psi_B(A) \cdot f(A). 
\] 
Any  $m\in M$ acts on  $\bar N$ by conjugation, and therefore on $\Omega \subset J$ via the identification $\bar N\cong J$. Then, for $f\in  L^2(\Omega)$, 
\[ 
\pi(m)f(A) = \chi(m)  f(m^{-1} A) 
\] 
where $\chi$ is a positive character of $M$ that we shall not need. 
 We have the following, see Propositions 7.2 and 8.3 in  \cite{KS15}: 

\begin{theorem} \label{T:minimal} 
 Let $\Pi$ be the subspace of $G$-smooth vectors in the minimal representation. Then $C_c^{\infty} (\Omega) \subset \Pi \subset C^{\infty}(\Omega)$. If $A\in J$ is 
non-zero, then any continuous functional $\ell$ on $\Pi$ such that $\ell( \pi(n) f) = \psi_A(n) \cdot \ell(f)$ for all $n\in N$ and $f\in \Pi$ is equal to a multiple of  the delta 
functional 
\[ 
f\mapsto f(A).
\]  
In particular, $\ell=0$ if $A$ is not rank one. If $F$ is $p$-adic, we moreover have an exact sequence of $P$-modules 
\[ 
0 \rightarrow C_c^{\infty} (\Omega) \rightarrow \Pi \rightarrow \Pi_N \rightarrow 0. 
\] 
 
\end{theorem} 

The representation $\Pi$ is spherical, and we describe a spherical vector in the $p$-adic case. Let $O$ be the ring of integers in $F$, 
$\varpi\in O$ a uniformizing element and $q$ the order of the residual field. The maximal order $\mathcal R$ in $\Oct$ defines an integral 
structure on $J$, let $J(O)$ be the lattice of $O$-points in $J$. 
 The greatest common divisor of entries of $A\in J(O)$, is the largest power $\varpi^n$ dividing $A$, that is, such that $A/\varpi^n$ is in $J(O)$.  
 We have the following Theorem 6.1 in \cite{SW07}:

\begin{theorem} \label{T:spherical} 
Assume $F$ is a $p$-adic field. 
 Assume the conductor of $\psi$ is $O$.  Then the spherical vector in $\Pi$ is a function $f^{\circ}\in C^{\infty}(\Omega)$ supported in $J(O)$. 
 Its value at $A\in \Omega$ depends on the gcd of entries of $A$. More precisely, if the gcd of $A$ is $\varpi^n$, and $q$ is the order of the residual field, then 
 \[ 
 f^{\circ} (A)= 1+ q^3 + \ldots + q^{3(n-1)}. 
 \] 
 \end{theorem} 
 
 \subsection{Local non-vanishing} 
 
 In this section $F$ is a $p$-adic field. Let $U_3 \subset M_3$ be the maximal unipotent subgroup that, via the isomorphism $M_3\cong \GL_3$ given previously, 
  is the group of matrices 
 \[ 
u=\left(\begin{array}{ccc} 
1 & a &  c\\
0 & 1 & b \\
0 & 0 & 1
\end{array}\right). 
\] 
Then $U=U_3 N_3$ is a maximal unipotent subgroup of  $\PGSp_6$. We define a Whittaker character  $\psi_U$ on $U$ as 
follows. For every $u\in U_3$, written as above, $\psi_U(u)=\psi(a+b)$. For every  $u\in N_3$, which we identify with $B\in J_3\cong N_3$, 
$\psi_U(u)=  \psi_E(B)$ where 
\[ 
E= \left(\begin{array}{ccc} 
0 & 0 &  0\\
0 & 0 & 0 \\
0 & 0 & 1
\end{array}\right). 
\] 

The following result was stated in Proposition 17 of \cite{GS04} without a proof. We give details following  Theorem 7.1 in \cite{G}, where a version of this result was proved for groups over finite 
fields.

\begin{theorem}  \label{T:GG} 
Let $\Pi$ be the minimal representation of $G$,  $U$ the maximal unipotent subgroup of $\PGSp_6$ and $\psi_U$ the Whittaker 
character of $U$. Then there exists a maximal unipotent subgroup $U_0$ of $G_2$ and a Whittaker character $\psi_0$ of $U_0$ such that 
\[ 
\Pi_{U,\psi_U} \cong \mathrm{ind}_{U_0}^{G_2} \psi_0 . 
\] 
\end{theorem} 
\begin{proof} We shall compute $(U,\psi_U)$-coinvariants in two steps,  $(N_3, \psi_U)$-coinvariants, followed by $(U_3,\psi_U)$-coinvaraints. 
From Theorem \ref{T:minimal} it follows that 
\[ 
\Pi_{N_3,\psi_U} \cong C_c^{\infty} (\Omega)_{N_3,\psi_U} \cong C_c^{\infty} (\omega) 
\] 
where $\omega$ is the set of all rank one elements $A$ such that $\psi_A= \psi_E$ on $N_3\cong J_3$. This works out to all 
 \[ 
A=\left(\begin{array}{ccc} 
0 & x &  -y \\
-x & 0 & z \\
y & -z & 1
\end{array}\right)\in J 
\] 
where $x,y,z\in \mathbb O$ satisfying 
\[ 
\Tr(x)=\Tr(y)=\Tr(z)= 0, 
\] 
\[ 
x^2=y^2=z^2=0 , 
\] 
\[ 
xz=xy=0, yz=x. 
\] 
It is now useful to write $\omega=\omega'\cup \omega''$ where $\omega'$ is the open subset given by $x\neq 0$. We claim that $G_2$ acts transitively 
on $\omega'$. To that end, we shall prove that $x, y, z$ can be $G_2$-conjugated to the elements 
$s_1, t_2, t_3$. Conjugate  by an element in $G_2$ so that $x=s_1$. Since $\langle s_1\rangle^{\Delta} = \langle s_1,t_2,t_3\rangle$, the elements $y,z$ must be 
contained in this 3-dimensional space. 
Let $Q_1=L_1U_1$ be the maximal parabolic subgroup of $G_2$ stabilizing the line $\langle s_1\rangle$, with the Levi subgroup $L_1 \cong \GL(\langle t_2, t_3\rangle)$. 
The stabilizer of $s_1$ in $G_2$  is 
$L_1^{\rm der} U_1$ where $L_1^{\rm der}\cong {\rm SL}(\langle t_2, t_3\rangle)$. 
The octonion multiplication gives a skew-symmetric form on $\langle s_1,t_2,t_3\rangle/\langle s_1\rangle$ 
valued in $\langle s_1\rangle$. Thus we can conjugate $y,z$ by an element in $L_1^{\rm der}$ to
\[ 
y=t_2 +as_1 \text{ and } z=t_3 + bs_1 
\] 
for $a,b\in F$. We still have the unipotent radical $U_1$ to work with. The derived subgroup $[U_1,U_1]$ acts trivially on $ \langle s_1,t_2,t_3\rangle$ and 
the 2-dimensional quotient $U_1/[U_1,U_1]$ acts simply transitively on the pairs $(a,b)\in F^2$ as above. This proves the claim. Note that we have also proved that 
$x,y,z$ are linearly independent if $x\neq 0$.  The group $U_3\subset M_3$ acts on $\omega$, and hence on the triples $(x,y,z)$ of off-diagonal terms. 
Explicitly this action is
\[ 
u^{-1}  \cdot (x,y,z) = (x, y+bx, z+ ay + cx). 
\] 
Hence $U_3$ acts freely on $\omega'$, by the linear independence of $x,y,z$, but with large stabilizers on $\omega''$,  assuring that 
\[ 
\Pi_{U,\psi_U} \cong  C_c^{\infty} (\omega')_{U_3, \psi_U}. 
\] 
Summarizing,  the situation is very pleasant: $G_2$ acts transitively on $\omega'$, while $U_3$ acts freely. View $\omega'$ as the $G_2$-orbit of $A_0$ corresponding 
to the triple $(x,y,z)=(s_1,t_2, t_3)$. The triple defines a partial flag 
\[ 
\langle s_1 \rangle \subset \langle s_1, t_2\rangle \subset \langle s_1, t_2, t_3\rangle. 
\] 
Let $U_0$ be the unipotent radical of the Borel subgroup in $G_2$ stabilizing this flag.  Observe that any $u_0\in U_0$ necessarily acts 
on the triple $(s_1, t_2, t_3)$ as an element $u\in U_3$.  In fact, this action can be made explicit for elements in simple root spaces of $U_0$  since they are contained 
in the Levi factors of the two maximal parabolic subgroups ($\GL(\langle s_1, t_2\rangle)$ and $\GL(\langle t_2, t_3\rangle)$): 
\[ 
(s_1, t_2,t_3) \mapsto (s_1, t_2 +\lambda \cdot s_1,t_3) \text{ and } (s_1, t_2,t_3) \mapsto (s_1,  t_2, t_3 +\lambda \cdot t_2 ) 
\] 
where $\lambda\in F$. Hence we 
have a surjective homomorphism $\varphi : U_0 \rightarrow U_3$, such that the pullback $\bar \psi_{0}$ of $\psi_U$ is a Whittaker 
functional on $U_0$. It follows that 
\[ 
\Pi_{U.\psi_U} \cong \ind_{U_0}^{G_2} (\psi_0). 
\] 
\end{proof} 

\subsection{Global non-vanishing} 
Assume now that $F$ is a global field, and  let $\mathbb A$ be the ring of ad\`eles over $F$. Let 
$\Pi=\otimes \Pi_v$ be the restricted tensor product of minimal representations over all local places $v$ of $F$. 
Every element in $\Pi$ is a finite linear combination of pure tensors $f= \otimes f_v$, where $f_v=f_v^{\circ}$ for almost all places $v$. 
There is a unique, up to a non-zero scalar, 
embedding 
\[ 
\theta: \Pi \rightarrow \mathcal A(G(F) \backslash G(\mathbb A)) 
\] 
 of $\Pi$ into the space of automorphic functions of uniform moderate growth.  We Fourier expand $\theta(f)$ along $N$. 
 More precisely, fix a non-trivial character $\psi: \mathbb A/ F\rightarrow \mathbb C^{\times}$. 
Then, as in the local case, any $A\in J(F)$ defines a character 
$\psi_A$ of $N(F)\backslash N(\mathbb A)$ by the isomorphism $N(\mathbb A) \cong J(\mathbb A)$.  Let 
\[ 
\theta(f)_A(g)  = \int_{N(F)\backslash N(\mathbb A)} \theta(f)(ng) \bar{\psi}_A(n)  ~dn. 
\] 
We have a Fourier expansion 
\[ 
\theta(f)(g)= \theta(f)_0(g) + \sum_{A\in \Omega(F)} \theta(f)_A(g) 
\]  
supported on the set of rank one elements.  Let $A\in \Omega(F)$. Observe that $f \mapsto \theta(f)_A(1)$ is a continuous, $(N(\mathbb A), \psi_A)$-equivariant functional on $\Pi$. 
 By uniqueness of local functionals,  Theorem \ref{T:minimal},  there exists a non-zero scalar $c_A$ such that $\theta_A(f)(1) =c_A \cdot f(A)$, for all $f\in \Pi$. Hence 
\[ 
\theta(f)_A(g)= c_A \cdot  (\pi(g)  f)(A) 
\] 
for all $f\in \Pi$ and all $g\in G(\mathbb A)$, where 
$\pi$ denotes the action of $g\in G(\mathbb A)$ on $f\in \Pi$.  
This formula is particularly useful if $g\in G_2(\mathbb A)$. Then $(\pi(g) f)(A)= f(g^{-1} A)$, where 
$g^{-1} A$ is the result of the natural action of $g^{-1}$ on the off-diagonal entries of $A$.

\smallskip

For every cusp form $h\in \mathcal A(G_2(F)\backslash G_2(\mathbb A))$ consider the function 
$\Theta(f,h)$ on $\PGSp_6(\mathbb A)$ defined by 
\[ 
\Theta(f,h)(g_1)= \int_{G_2(F)\backslash G_2(\mathbb A)} \theta(f)(g_1 g) h(g) ~dg. 
\] 
The function $\theta(f)$ is of moderate growth on $G$, hence it is also on $G_2(\mathbb A)\times \PGSp_6(\mathbb A)$. 
In particular, the integral converges absolutely, since $h$ is a cusp form, 
and the output $\Theta(f,h)$ is a function of uniform moderate growth on $\PGSp_6(\mathbb A)$.   
Let $U_0$ be the maximal unipotent subgroup stabilizing the partial flag 
$\langle s_1 \rangle \subset \langle s_1, t_2\rangle \subset \langle s_1, t_2, t_3\rangle $.  Let $U'_0\subset U_0$ be the stabilizer of the triple $(s_1, t_2, t_3)$. 
Then $U_0/U'_0$ is isomorphic to $U_3$ and, via this isomorphism, the Whittaker character $\psi_U$ of $U(F)\backslash U(\mathbb A)$ transfers to a Whittaker character $\psi_0$ 
of $U_0(F)\backslash U_0(\mathbb A)$  as in the local case. 
Let 
\[ 
h_{U_0, \psi_0}(g) = \int_{ U_0(F)\backslash U_0(\mathbb A)} h(u g) \bar{\psi}_0(u) ~du. 
\] 

\begin{theorem}\label{thm:main-appendix} Let  $\psi_U$ be the Whittaker character of $U(F)\backslash U(\mathbb A)$. If $h_{U_0, \psi_0}\neq 0$ then, for some choice of $f\in \Pi$, 
\[ 
\Theta(f,h)_{U,\psi_U}(1):=\int_{U(F)\backslash U(\mathbb A)} \Theta(f,h)(u)   \bar{\psi}_U(u) ~du \neq 0 .
\] 
\end{theorem} 
\begin{proof} The first part of the proof involves using the Fourier expansion of $\theta(f)$ and unfolding the integral. 
This follows closely the proof of the local Theorem \ref{T:GG}, and we shall be brief. 
Firstly we integrate over $N_3(F)\backslash N_3(\mathbb A)$. This reduces the Fourier sum to over the subset $\omega(F)$ 
\[ 
\Theta(f,h)_{N_3,\psi_U}(1) = \int_{G_2(F)\backslash G_2(\mathbb A)} \sum_{A\in \omega(F)} \theta(f)_A(g) h(g) ~dg. 
\] 
Write $\omega=\omega'\cup\omega''$ as in the local case. We can ignore the sum over $\omega''$ since it will vanish after integrating over 
$U_3(F)\backslash U_3(\mathbb A)$. On the other hand, $\omega'(F)$ is the $G_2(F)$-orbit of $A_0$ corresponding to the triple  $(x,y,z)=(s_1,t_2,t_3)$ 
with the stabilizer $U'_0(F)$.   Then 
\[ 
\int_{G_2(F)\backslash G_2(\mathbb A)} \sum_{A\in \omega'(F)} \theta(f)_A(g) h(g) ~dg =
c_{A_0} \int_{U'_0(F)\backslash G_2(\mathbb A)} f(g^{-1} A_0)  h(g) ~dg. 
\] 
We integrate over $U'_0(F)\backslash U'_0(\mathbb A)$,  and use that the function $g\mapsto f(g^{-1} A_0)$ is left $U'_0(\mathbb A)$-invariant, hence
\[ 
c_{A_0} \int_{U'_0(F)\backslash G_2(\mathbb A)} f(g^{-1} A_0)  h(g) ~dg = 
c_{A_0} \int_{U'_0(\mathbb A)\backslash G_2(\mathbb A)} f(g^{-1} A_0)  h_{U'_0}(g) ~dg 
\] 
where $h_{U'_0}$ is the constant term of $h$ along $U'_0$. Finally, we integrate over 
 over $U_3(F)\backslash U_3(\mathbb A)$ against the character $\bar \psi_U$. Using the isomorphism $U_3\cong U_0/U'_0$, we obtain 
\[ 
\Theta(f,h)_{U,\psi_U}(1)= c_{A_0} \int_{U'_0(\mathbb A)\backslash G_2(\mathbb A)} \tilde f (g) h_{U'_0}(g)   ~dg 
\] 
where $\tilde f$ is the product of 
\[ 
\tilde f_v(g) =\int_{ U_0'(F_v)\backslash U_0(F_v)} f_v((ug)^{-1}  A_0) \psi_0(u)  ~du. 
\] 
Observe that $\tilde f$ is a Whittaker function i.e. $\tilde f(ug)= \bar\psi_0(u) \tilde f(g)$, for all $u\in U_0(\mathbb A)$. Hence 
\[ 
\Theta(f,h)_{U,\psi_U}(1)= c_{A_0} \int_{U_0(\mathbb A)\backslash G_2(\mathbb A)} \tilde f (g) h_{U_0,\psi_0}(g)   ~dg. 
\]

\smallskip 
The next step is to show that this integral reduces to a finite number of places $S$, as in Section 7 of \cite{GS03}. Assume that 
$S$ contains all archimedean places and, if $v\notin S$, then  
\begin{itemize} 
\item $f_v$ is the spherical vector $f_v^{\circ}$,
\item the cusp form $h$ is $G_2(O_v)$-invariant, 
\item the conductor of $\psi$ restricted to $F_v$ is $O_v$. 
\end{itemize} 

 Let  $B_0$ be the Borel subgroup containg $U_0$. We fix a Levi 
factor $L_0$ (a torus) so that is stabilizes the lines through $s_1,t_2, t_3$. Thus, if  $l_v\in L_0(F_v) $ then 
\[ 
l_v t_2= \lambda_2\cdot  t_2, ~ l_v  t_3= \lambda_3 \cdot t_3 ,  ~ l_v  s_1= (\lambda_2\lambda_3) \cdot  s_1 , 
\] 
for some non-zero scalars $\lambda_2$ and $\lambda_3$, where the last identity follows from $s_1=t_2 t_3$.   

\begin{lemma} If $v\notin S$ then the Whittaker function  $\tilde f_v^{\circ}$ is supported on $U_0(F_v) G_2(O_v)$ and $\tilde f_v^{\circ}(1)=1$. 
\end{lemma} 
\begin{proof} Let $\ord_v$ denote the valuation on $F_v^{\times}$. Since $\tilde f_v^{\circ}$ is a spherical Whittaker function, 
 it is determined by its restriction to $L_0(F_v)$ and there it 
 is supported on the cone consisting of $l_v$  such that 
 $\ord_v (\alpha_1(l_v)) \geq 0$ and $\ord_v(\alpha_2(l_v))\geq 0$ where $\alpha_1$ and $\alpha_2$ are the simple roots. In terms of the explicit 
action of $l_v$ on $s_1, t_2, t_3$ described above,  these inequalities are 
\[ 
\ord_v( \lambda_2 ) \geq \ord_v( \lambda_3) \geq 0. 
\] 
On the other hand, since $f_v^{\circ}$ is supported on the lattice $J(O_v)$,   it is easily seen that, for any $u_v\in U_0(F_v)$, 
the function $l_v\mapsto f_v^{\circ}( (u_v l_v)^{-1} A_0)$ is supported on the cone 
\[
\ord_v( \lambda_2 ),  \ord_v( \lambda_3) \leq  0.
\] 
This completes the support part of the statement. Using again that $f_v^{\circ}$ is supported on the lattice $J(O_v)$, the  function $u_v\mapsto f_v^{\circ}( u_v^{-1} A_0)$, is supported 
on $U'_0(F_v) U_0(O_v)$ and this implies that $\tilde f_v^{\circ}(1)=1$, since the character $\psi_0$ is trivial on $U_0(O_v)$. 

\end{proof} 

Let $\mathbb A_S$ be the product of $F_v$ over all $v\in S$. The previous lemma implies that 
\[ 
\Theta(f,h)_{U,\psi_U}(1)= c_{A_0} \int_{U_0(\mathbb A_S)\backslash G_2(\mathbb A_S)} \tilde f_S(g) h_{U_0,\psi_0} (g)  ~dg.   
\] 
Since, by Theorem \ref{T:minimal},  $f_v$ can be any smooth, compactly supported function on $\Omega(F_v)$, for every $v\in S$, the integral can be arranged to be non-zero. 
This completes the proof of the theorem.

\end{proof} 

\noindent 
{\bf Remark:} Note that  $\Theta(f,h)$ is unramified for all $v\notin S$.

\end{document}